\documentclass[reqno]{amsart}

\usepackage{amsmath,amssymb,amsthm,amsfonts}
\usepackage{hyperref}
\usepackage{mathrsfs}
\usepackage{appendix}
\usepackage{graphicx}
\usepackage{setspace}

\usepackage{color}

\newtheorem{lemma}{Lemma}[section]
\newtheorem{theorem}{Theorem}[section]

\newtheorem{proposition}{Proposition}[section]
\newtheorem{remark}{Remark}[section]

\numberwithin{equation}{section}

\newcommand{\beq}{\begin{equation}}
\newcommand{\eeq}{\end{equation}}
\newcommand{\ben}{\begin{eqnarray}}
\newcommand{\een}{\end{eqnarray}}
\newcommand{\beno}{\begin{eqnarray*}}
\newcommand{\eeno}{\end{eqnarray*}}

\allowdisplaybreaks

\begin{document}

\title[Stability for the 2-D Poiseuille flow]{Enhanced dissipation and Transition threshold for the 2-D plane Poiseuille flow via Resolvent Estimate}




\author[Shijin Ding, Zhilin Lin]{Shijin Ding, Zhilin Lin$^*$}

\address[S. Ding]{South China Research Center for Applied Mathematics and Interdisciplinary Studies, South China Normal University,
Guangzhou, 510631, China}

\address{School of Mathematical Sciences, South China Normal University,
Guangzhou, 510631, China}
\email{dingsj@scnu.edu.cn}

\address[Z. Lin]{School of Mathematical Sciences, South China Normal University,
Guangzhou, 510631, China}
\email{zllin@m.scnu.edu.cn}

\thanks{$^*$Corresponding author}
\date{\today}

\begin{abstract}
In this paper, we study the transition threshold problem for the 2-D Navier-Stokes equations around the Poiseuille flow $(1-y^2,0)$ in a finite channel with Navier-slip boundary condition. Based on the resolvent estimates for the linearized operator around the Poiseuille flow, we first establish the enhanced dissipation estimates for the linearized Navier-Stokes equations with a sharp decay rate $e^{-c\sqrt{\nu}t}$. As an application, we prove that if the initial perturbation of vortiticy satisfies
$$\|\omega_0\|_{L^2}\leq c_0\nu^{\frac{3}{4}},$$
for some small constant $c_0>0$ independent of the viscosity $\nu$, then the solution dose not transition away from the Poiseuille flow for any time.
\end{abstract}

\maketitle

\section{Introduction}

In this paper, we study the transition threshold problem for the 2-D incompressible Navier-Stokes equations in a channel $\Omega=\mathbb{T} \times I(I=(-1,1))$:
\begin{equation}\label{1.1}
\left \{
\begin{array}{lll}
\partial_t v-\nu \Delta v+(v\cdot \nabla)v +\nabla q=0,\\
\nabla \cdot v=0,\\
v|_{t=0}=v_0(x,y),
\end{array}
\right.
\end{equation}
where $\nu >0$ is the viscosity, $v(t;x,y)\in \mathbb{R}^2$ is the velocity fields and $q\in \mathbb{R}$ is the pressure.

It is well known that the Poiseuille flow $V=(1-y^2,0)$ is a steady solution of the Navier-Stokes equations in (\ref{1.1}) with a constant pressure gradient $\nabla Q \equiv (Q_0,0)$ for some constant $Q_0$. In order to study the stability of the Poiseuille flow, we introduce the perturbations $u=v-V$ and $P=q-Q$, which satisfy
\begin{equation}\label{1.2}
\left \{
\begin{array}{lll}
\partial_t u-\nu \Delta u+(u\cdot \nabla)u +(1-y^2)\partial_x u+\left(
\begin{array}{c} -2y u_2 \\ 0 \end{array}\right)+\nabla P =0,\\
\nabla \cdot u=0,\\
u|_{t=0}=u_0(x,y,z).
\end{array}
\right.
\end{equation}

To avoid the boundary layer effect, we will consider the Navier-slip boundary condition for the perturbation system \eqref{1.2}:
\begin{equation}\label{1.2a}
\omega =0\quad \mathrm{on}\ \ \{y=\pm1\}.
\end{equation}
Here $\omega=\partial_y u_1-\partial_x u_2$ is the vorticity.

Introduce the stream function $\Phi$ as follows
$$ u=\nabla^\bot \Phi=(\partial_y \Phi,-\partial_x \Phi),$$
then $\Delta \Phi=\omega$ and \eqref{1.2} with \eqref{1.2a} can be rewritten as
\begin{equation}\label{1.2d}
\left \{
\begin{array}{lll}
\partial_t \omega-\nu\Delta \omega+(1-y^2)\partial_x \omega+2\partial_x \Delta^{-1}\omega =\nabla\cdot(\omega u),\\
\Delta \Phi=\omega,\\
\omega(t;x,\pm1)=0,\\
\omega|_{t=0}=\omega_0=:\omega(0).
\end{array}
\right.
\end{equation}

The hydrodynamic stability at high Reynolds number (or small viscosity) has been an important topic since the Reynolds's famous experiment \cite{Rey}. The problem that how the laminar flows become unstable and transition into turbulence is a very important topic in this field, see \cite{Sch,Yaglom} for instance. It is well-known that some laminar flows such as plane Couette flow and pipe Poiseuille flow are linearly stable for any fixed Reynolds number (or viscosity) \cite{Rom,CWZ1}. However, the instability even turbulence would occur in experiments with small perturbations at high Reynolds number \cite{DHB}, which is contradictory with theoretical analysis. This phenomena is well known as Sommerfeld paradox and it is a longstanding problem in the theory of hydrodynamic stability, see \cite{Li} and the references therein.

There are lots of attempts from different points of view to resolve this paradox, such as \cite{Chapman,Li} and the references therein. Among these literatures, the resolution introduced by Kelvin \cite{Kelvin} is that the basin of attraction of laminar flows shrinks as the Reynolds number tends to infinity such that the flow could be nonlinearly unstable for small perturbations. With this resolution, a fundamental question firstly proposed by Trefethen et al. \cite{Tre}, formulated as the mathematical version in \cite{BGM4} is that\smallskip

\emph{Given a norm $\Vert \cdot\Vert_X$, determine a $\gamma=\gamma(X)$ such that
$$\Vert u_0\Vert_X \leq \nu^\gamma \Longrightarrow \ \ stability,$$
$$\Vert u_0\Vert_X \gg\nu^\gamma \Longrightarrow\ \ instability.$$}

\noindent Here the exponent $\gamma$ is referred to as the transition threshold in the applied literatures, and there are lots of applied works devoted to determining $\gamma$ for some important laminar flows such as Couette flow and Poiseuille flow \cite{LHR,Reddy,Yaglom}. \smallskip

Recently, the transition threshold problem for the Couette flow was studied in a series of important mathematical works \cite{BGM1,BGM2,BGM3,BMV,BVW,WZ}.
These works showed that the transition threshold $\gamma\le 1$ with either $X$ Gevrey class or Sobolev space for the 3-D Couette flow in the case without the boundary effect. Very recently, Chen, Li, Wei and Zhang \cite{CLWZ} studied the transition threshold problem for the 2-D Couette flow in a finite channel. Moreover, Chen, Wei and Zhang \cite{CWZ2} proved that the transition threshold $\gamma \leq 1$ still holds in the Sobolev space for the 3-D Couette flow in a finite channel with nonslip boundary condition for the perturbation.

For the 2-D Kolmogorov flow, Wei, Zhang and Zhao \cite{WZZ3} showed that the transition threshold is $\gamma \leq \frac{2}{3}+$ in $\mathbb{T}_{2\pi\delta}\times \mathbb{T}$ with $\delta<1$. For the 3-D Kolmogorov flow, Li, Wei and Zhang \cite{LWZ2} proved that the the transition threshold is $\gamma \leq \frac{7}{4}$, where the proof is based on the resolvent estimate method and wave operator method developed in \cite{LWZ1}. For the Poiseuille flow $(y^2,0)$, Coti Zelati et al. \cite{Coti} proved that the transition threshold is $\gamma \leq \frac{3}{4}+2\mu$ (for any $\mu>0$) in $\mathbb{T}\times \mathbb{R}$.

For general shear flows, the stability and transition problems would be very challenging, since the linearized operator is non selfadjoint and nonlocal. In the inviscid case $\nu=0$, the analysis for the 2D linearized problem is reduced to solving the Rayleigh equations \cite{WZZ1,WZZ2,WZZ3}. However, in the viscous problem with small $\nu>0$, the problem would be much difficult since one has to solve the Orr-Sommerfeld equation, which is a fourth-order ODE. For example, for the Orr-Sommerfeld equation around a non-monotone flow such as Poiseuille flow in this paper, one has to study the behavior near the critical point and handle nonlocal term carefully. There are some important progress on the enhanced dissipation for the 2-D linearized Navier-Stokes equations, see \cite{Beck,Gall,IMM,LWZ2,LX,WZZ3} for details. We also refer to the survey article \cite{BGM4} for more results and open problems.\smallskip

The first result of this paper is to establish the enhanced dissipation estimates for the linearized Navier-Stokes equations around the 2-D Poiseuille flow:
\begin{equation}\label{1.7}
\left \{
\begin{array}{lll}
(\partial_t+\mathscr{L}) \omega =0,\\
\omega(t;x,\pm1)=0,\\
\omega|_{t=0}=\omega_0,
\end{array}
\right.
\end{equation}
where
\begin{align}
\label{1.5}
&\mathscr{L} \omega=-\nu\Delta \omega+(1-y^2)\partial_x \omega+2\partial_x \Delta^{-1}\omega.
\end{align}

Let $\widehat{f}(t;k,y)$ denote the Fourier transform of the function $f(t;x,y)$ with respect to $x$. And define
$$\overline{f}(t;x,y):=\int_{\mathbb{T}}f\mathrm{d}x, \ f_{\not=}:=f-\overline{f}.$$

First, we have the following enhanced dissipation with sharp decay rate for linearized Navier-Stokes equations.
\begin{theorem}\label{thm:ED}
For any $0<\nu \leq 1$, if $\omega_0\in L^2$ with $\int_{\mathbb{T}}\omega_0(x,y)\mathrm{d}x=0,$ then the solution of \eqref{1.7} satisfies the following estimates
\begin{equation}\nonumber
\begin{aligned}
\|\omega(t)\|_{L^2}\leq C e^{-c\sqrt{\nu}t}\|\omega_0\|_{L^2}.
\end{aligned}
\end{equation}
\end{theorem}

As an application, the second result of this paper is to study the stability threshold for the 2-D Poiseuille flow, which is stated as follows.
\begin{theorem}\label{mainresult1}
There exist $\nu_0\in (0,1)$ and $c_0\in (0,1)$ independent of $\nu$ so that if $\nu\in (0,\nu_0)$ and $\omega_0\in L^2$ with
$$\Vert \omega_0\Vert_{L^2}\leq c_0\nu^{\frac{3}{4}},$$
then the solution $\omega$ of (\ref{1.2d}) is global in time with the bounds
\begin{equation}\nonumber
\begin{aligned}
\|\omega_{\not=}(t)\|_{L^2}\leq C e^{-c\sqrt{\nu}t}\|\omega_{\not=}(0)\|_{L^2}.
\end{aligned}
\end{equation}

\end{theorem}

\begin{remark}\label{rk1}
Our main results, the Theorem \ref{thm:ED} and Theorem \ref{mainresult1}, give the enhanced dissipation with decay rate $e^{-\sqrt{\nu}t}$ and the transition threshold $\gamma\leq \frac{3}{4}$, respectively.
In \cite{Coti}, the enhanced dissipation with decay rate $e^{-\varepsilon_0\frac{\sqrt{\nu}}{1+|\log \nu|}t}$ and $\gamma\leq \frac{3}{4}+2\mu$ for any $\mu>0$ for the 2-D Poiseuille flow $(y^2,0)$ in the domain $\mathbb{T}\times \mathbb{R}$ were firstly obtained via the hypocoercivity method.

Compared with \cite{Coti}, our problem is to treat with the case in a channel $\mathbb{T}\times (-1,1)$ and our method is very different from that in \cite{Coti}.
Our proof is based on the resolvent estimates, which give the enhanced dissipation with decay rate $e^{-c\sqrt{\nu}t}$, and hence the transition threshold $\gamma\leq  \frac{3}{4}.$
\end{remark}
\begin{remark}\label{rk2}
The transition threshold $\gamma\leq \frac{3}{4}$ may not be sharp. To improve it, much substantial works are needed.

In addition, the transition threshold $\gamma\leq \frac{7}{4}$ for the 3-D Kolmogorov flow is obtained in \cite{LWZ2}. It was conjectured in \cite{Chapman} that the transition threshold $\gamma\leq \frac{3}{2}$ for the 3-D plane Poiseuille flow.
\end{remark}

Let us give some comments about the main results. Two main theorems give the linear enhanced dissipation and nonlinear stability transition threshold for 2-D Poiseuille flow. Our first result gives the enhanced dissipation with a sharp decay rate $e^{-c\sqrt{\nu}t}$, which is based on our careful resolvent estimates for the O-S equation. This is very different from the hypocoercivity method applied in \cite{Coti}.
We study the enhanced dissipation and transition threshold problem for the 2-D Poiseuille flow $(1-y^2,0)$ in a channel $\mathbb{T}\times (-1,1)$, which is very different from the case of $\mathbb{T}\times \mathbb{R}$ in \cite{Coti}.
The key point is to derive the sharp resolvent estimates for the O-S equations with boundaries.
Due to the interval $(-1,1)$, we will derive the resolvent estimate by several cases. The most challenges are resulted from the critical points of Poiseuille flow in the interval $(-1,1)$, in which more estimates near the points are needed. To overcome the difficulties, more careful estimates about some suitable multipliers near the critical points are needed.
Based on the careful and sharp resolvent estimates, the enhanced dissipation with sharp decay rate $e^{-c\sqrt{\nu}t}$ and transition threshold $\nu^\gamma$ for $\gamma\leq\frac{3}{4}$ are obtained.

Through this paper, we always suppose that $|k|\geq1$, and denote by $C$ a positive constant independent of $\nu,k,\lambda$. In addition, we use the notation $a \sim b$ for $C^{-1}b \leq a \leq Cb$ and $a \lesssim b$ for $a \leq Cb$, in which $C>0$ is an absolute constant.

\section{Sketch and key points of the proof}\label{sec2}
Our proof is based on the careful resolvent estimates. It is necessary to give some analysis on our problem and method.

We first study the resolvent estimates of the following
linearized operator
\beno
&-\nu(\partial_y^2-|k|^2)w+ik\big[(1-y^2-\lambda)w+2\varphi\big]=F,
\eeno
where $(\partial_y^2-|k|^2)\varphi=w, \,\varphi(\pm 1)=0$.
In Proposition \ref{prop:res}, we will establish the following resolvent estimate:
\beno
\nu^{\frac{1}{2}}|k|^{\frac{1}{2}}\Vert w \Vert_{L^2}\leq C\Vert F\Vert_{L^2}.
\eeno
The proof follows the resolvent estimate method introduced in \cite{LWZ2}.

The key point in our problem is to derive the resolvent estimates for O-S equation with boundaries. The estimates will be concluded by several cases: (i) $\lambda \geq 1$; (ii) $\lambda \leq 0$; (iii) $\lambda \in (0,1)$. The most difficult case is (iii), since one has to deal with the estimates near the points $y_i$ with $1-y_i^2=\lambda (i=1,2)$. Therefore, we should establish the estimates on the interval $(y_1,y_2)$ and outside interval, respectively. To this end, more careful estimates about some suitable multipliers near the points $y_i$ should be needed. See Section \ref{secresolvent} for details.

Taking the Fourier transformation in $x$, the linearize operators $\mathscr{L}$ can be rewritten as
\beno
&\widehat{\mathscr{L}}=-\nu(\partial_y^2-|k|^2)+ik(1-y^2)+2ik(\partial_y^2-|k|^2)^{-1},
\eeno
Unlike the case of Kolmogorov flow \cite{LWZ2}, the operators $\widehat{\mathscr{L}}$ is symmetric and m-accretive in the sense
\begin{equation}\nonumber
\mathrm{Re} \big\langle \widehat{\mathscr{L}}f,f\big\rangle =\nu |k|^2\Vert f\Vert_{L^2}^2+\nu\Vert f' \Vert_{L^2}^2 \geq 0.
\end{equation}
We define the pseudospectral bound for an accretive operator $A$ by
$$\Psi(A):=\inf\big\{\Vert (A-i\lambda)s\Vert:s\in D(A),\lambda \in \mathbb{R},\Vert s\Vert=1\big\}.$$
Based on the resolvent estimates for $\widehat{\mathscr{L}}$, we can deduce that
\begin{equation}\nonumber
\begin{aligned}
\Psi(\widehat{\mathscr{L}}) \geq c\nu^{\frac{1}{2}}|k|^{\frac{1}{2}}+\nu |k|^2.
\end{aligned}
\end{equation}
Then we can easily obtain the sharp semigroup bound:
$$\Vert e^{-t\mathscr{L}}g_{\not=}\Vert_{L^2}\leq Ce^{-c\sqrt{\nu}t-\nu t}\Vert g_{\not=}\Vert_{L^2}.$$

To derive the enhanced dissipation estimates for the linearized Navier-Stokes equations, we consider the following equation
\begin{equation}\nonumber
\left \{
\begin{array}{lll}
(\partial_t+\widehat{\mathscr{L}})\widehat{\omega}=0,\\
\widehat{\omega}(t;x,\pm 1)=0,\\
\widehat{\omega}|_{t=0}=\widehat{\omega}_0.
\end{array}
\right.
\end{equation}
Based on the sharp semigroup bound, the enhanced dissipation is obtained.

To study nonlinear problem, we first establish the space-time estimates for the following coupled system
\begin{equation}\nonumber
\left \{
\begin{array}{lll}
\big(\partial_t+\mathscr{L}\big)\omega=\nabla\cdot f,\\
\omega(t;x,\pm 1)=0,\\
\omega|_{t=0}=\omega_0=:\omega(0).
\end{array}
\right.
\end{equation}
More precisely, the following estimate will be established
\begin{equation}\nonumber
\begin{array}{lll}
\Vert  \omega_{\not=}\Vert_{X_{c'}}^2\lesssim \Vert \omega_{\not=}(0)\Vert_{L^2}^2+\nu^{-1}\Vert e^{c'\sqrt{\nu}t} f_{\not=}\Vert_{L^2L^2}^2,
\end{array}
\end{equation}
where
$$\Vert \omega_{\not=}\Vert_{X_{c'}}^2=\Vert e^{c'\sqrt{\nu}t}\omega_{\not=}\Vert_{L^\infty L^2}^2+\sqrt{\nu}\Vert e^{c'\sqrt{\nu}t}\omega_{\not=}\Vert_{L^2 L^2}^2+\nu\Vert e^{c'\sqrt{\nu}t}\nabla \omega_{\not=}\Vert_{L^2 L^2}^2.$$

Let us remark that the first two parts in the norm $\Vert \cdot \Vert_{X_{c'}}$, correspond to the enhanced dissipation, whose decay rate $e^{-c\sqrt{\nu}t}$ is resulted from the sharp resolvent estimates, and the last part is due to the combined effect of heat diffusion and enhanced dissipation.

Then we will obtain nonlinear stability from a continuity argument.

\section{Resolvent estimates for Orr-Sommerfeld equation}\label{secresolvent}

In this section, we will establish the resolvent estimates of the Orr-Sommerfeld equation
\begin{equation}\label{eq:res}
\left \{
\begin{array}{lll}
-\nu(\partial_y^2-|k|^2)w+ik\big[(1-y^2-\lambda)w+2\varphi\big]=F,\\
w(\pm 1)=0,
\end{array}
\right.
\end{equation}
where
\beno
(\partial_y^2-|k|^2)\varphi=w,\quad \varphi(\pm 1)=0.
\eeno

The main result about the resolvent estimate is stated as the following proposition.
\begin{proposition}\label{prop:res}
Let $w \in H^2(I)$ be the solution of (\ref{eq:res}) with $\lambda \in \mathbb{R}$ and $F \in L^2(I)$. Then there holds that
\begin{equation}\nonumber
\begin{aligned}
\nu^{\frac{1}{2}}|k|^{\frac{1}{2}} \Vert w \Vert_{L^2}
\leq C \Vert F \Vert_{L^2}.
\end{aligned}
\end{equation}
\end{proposition}

\subsection{Some basic lemmas}

The following key energy inequality \eqref{eq:energy-key} takes advantage of the special structure of Poiseuille flows.

\begin{lemma}\label{lem:energy}
Let $-1 \leq y_1 \leq 0 \leq y_2 \leq 1$. If $(\varphi,w)$ solves
\begin{equation}\label{eq:stream-cut}
\left \{
\begin{array}{lll}
\big(\partial_y^2-|k|^2\big)\varphi=w, \ y \in[y_1,y_2],\\
\varphi(y_1)=\varphi(y_2)=0,
\end{array}
\right.
\end{equation}
then we have
\begin{equation}\label{eq:energy-cut1}
-\big\langle \varphi,w\chi_{(y_1,y_2)}\big\rangle=\Vert \varphi' \Vert_{L^2(y_1,y_2)}^2+|k|^2\Vert \varphi \Vert_{L^2(y_1,y_2)}^2,
\end{equation}
and
\begin{equation}\label{eq:energy-key}
\begin{aligned}
2\left(\int_{y_1}^{y_2} (1-y^2-\lambda)^2\left|\left(\frac{\varphi}{1-y^2-\lambda}\right)'\right|^2\mathrm{d}y+\int_{y_1}^{y_2} k^2 |\varphi|^2\mathrm{d}y\right) \\
\leq \int_{y_1}^{y_2} (1-y^2-\lambda)|w|^2\mathrm{d}y+\big\langle 2 \varphi,w \chi_{(y_1,y_2)}\big\rangle.
\end{aligned}
\end{equation}
If $(\varphi,w)$ solves
\begin{equation}\nonumber
\left \{
\begin{array}{lll}
\big(\partial_y^2-|k|^2\big)\varphi=w, \quad y \in [-1,1]\setminus (y_1,y_2),\\
\varphi(y_1)=\varphi(y_2)=\varphi(\pm 1)=0,
\end{array}
\right.
\end{equation}
then we have
\begin{equation}\label{eq:energy-cut2}
-\big\langle \varphi,w\chi_{\big([-1,1]\setminus (y_1,y_2)\big)}\big\rangle=\Vert \varphi' \Vert_{L^2\big([-1,1]\setminus (y_1,y_2)\big)}^2+|k|^2\Vert \varphi \Vert_{L^2\big([-1,1]\setminus (y_1,y_2)\big)}^2.
\end{equation}
\end{lemma}

\begin{proof}
(\ref{eq:energy-cut1}) and (\ref{eq:energy-cut2}) are obvious.
We get by integration by parts that
$$\int_{y_1}^{y_2}\left( |\varphi'|^2-\frac{2}{1-y^2-\lambda}|\varphi|^2\right) \mathrm{d}y=\int_{y_1}^{y_2} (1-y^2-\lambda)^2\left|\left(\frac{\varphi}{(1-y^2-\lambda)}\right)'\right|^2\mathrm{d}y,$$
from which and the fact that
$$(1-y^2-\lambda)|w|^2+4 \varphi w +4\frac{|\varphi|^2}{1-y^2-\lambda} \geq 0,$$
we infer that for $y\in (y_1,y_2)$,
\begin{equation}\nonumber
\begin{aligned}
&\int_{y_1}^{y_2}(1-y^2-\lambda)|w|^2\mathrm{d}y+\big\langle 2 \varphi,w\chi_{(y_1,y_2)}\big\rangle \\
&\geq -\big\langle 2 \varphi,w\chi_{(y_1,y_2)}\big\rangle -4\int_{y_1}^{y_2}\frac{|\varphi|^2}{1-y^2-\lambda}\mathrm{d}y\\
&=2\left(\int_{y_1}^{y_2}(|\varphi'|^2+|k|^2|\varphi|^2)\mathrm{d}y\right)-4\int_{y_1}^{y_2}\frac{|\varphi|^2}{1-y^2-\lambda}\mathrm{d}y\\
&\geq 2\int_{y_1}^{y_2}\Big(|\varphi'|^2-\frac{2}{1-y^2-\lambda}|\varphi|^2\Big)\mathrm{d}y+2 \int_{y_1}^{y_2} |k|^2|\varphi|^2\mathrm{d}y\\
&=2 \int_{y_1}^{y_2} (1-y^2-\lambda)^2\left|\left(\frac{\varphi}{(1-y^2-\lambda)}\right)'\right|^2\mathrm{d}y+2 \int_{y_1}^{y_2} |k|^2|\varphi|^2\mathrm{d}y,
\end{aligned}
\end{equation}
which gives (\ref{eq:energy-key}).
\end{proof}

In some sense, the inequality means that the plane Poiseuille flow is a stable steady solution of the Euler equations. \smallskip

The following lemma is a Hardy type inequality.

\begin{lemma}\label{lem:hardy}
For $\lambda\in [0,1]$, let $-1 \leq y_1 \leq 0\leq y_2 \leq 1$ so that $\lambda=1-y_1^2=1-y_2^2$. If $(\varphi,w)$ solves (\ref{eq:stream-cut}), then we have
\begin{equation}\nonumber
\begin{aligned}
&\int_{y_1}^{y_2}\frac{|\varphi|^2}{(1-y^2-\lambda)^2}\mathrm{d}y\\
&\lesssim \frac{1}{(y_2-y_1)^2}\int_{y_1}^{y_2}(1-y^2-\lambda)^2\left|\left(\frac{\varphi}{(1-y^2-\lambda)}\right)'\right|^2\mathrm{d}y+\frac{|\varphi(0)|^2}{(y_2-y_1)^3}.
\end{aligned}
\end{equation}
\end{lemma}

The proof is similar to Lemma 3.7 in \cite{LWZ2}. Here we omit the details.\smallskip

\begin{lemma}\label{lem:P}
For $\lambda\in [0,1]$, let $-1 \leq y_1 \leq 0\leq y_2 \leq 1$ so that $\lambda=1-y_1^2=1-y_2^2$. Then for any $0\leq \delta \leq 1$,
\begin{align}
&\frac{1}{|1-(y_1-\delta)^2-\lambda|}=\frac{1}{(y_2-y_1+\delta)\delta},\label{eq:P-eq}\\
&\left\Vert \frac{1}{1-y^2-\lambda}\right\Vert_{L^2\big((-1,1)\setminus (y_1-\delta,y_2+\delta)\big)}\lesssim \frac{1}{(y_2-y_1+\delta)\delta^{\frac{1}{2}}},\label{eq:P-L2}\\
&\left\Vert \frac{1}{1-y^2-\lambda}\right\Vert_{L^1\big((-1,1)\setminus (y_1-\delta,y_2+\delta)\big)}\lesssim \frac{1+\ln \left(1+\frac{y_2-y_1}{\delta}\right)}{(y_2-y_1+\delta)},\label{eq:P-L1}\\
&	\left\Vert \frac{2y}{(1-y^2-\lambda)^2}\right\Vert_{L^2\big((-1,1)\setminus (y_1-\delta,y_2+\delta)\big)}\lesssim \frac{1}{(y_2-y_1+\delta)\delta^{\frac{3}{2}}}.\label{eq:P-L2-w}
\end{align}
\end{lemma}

\begin{proof}
\eqref{eq:P-eq} follows by noticing that $y_1=-y_2$ and
$$1-(y_1-\delta)^2-\lambda=y_2^2-(y_1-\delta)^2=(y_2-y_1+\delta)\delta.$$
Now \eqref{eq:P-L2} follows from
\begin{equation}\nonumber
\begin{aligned}
\int_{y_2+\delta}^1\left|\frac{1}{1-y^2-\lambda}\right|^2\mathrm{d}y=&\int_{y_2+\delta}^1\left|\frac{1}{(y-y_1)(y-y_2)}\right|^2\mathrm{d}y\\
\lesssim&\frac{1}{(y_2-y_1+\delta)^2}\int_{y_2+\delta}^1 \frac{1}{(y-y_2)^2}\mathrm{d}y\\
\lesssim &\frac{1}{(y_2-y_1+\delta)^2\delta},
\end{aligned}
\end{equation}
and \eqref{eq:P-L1} follows from
\begin{equation}\nonumber
\begin{aligned}
\int_{y_2+\delta}^1\left|\frac{1}{1-y^2-\lambda}\right|\mathrm{d}y=&\int_{y_2+\delta}^1\left|\frac{1}{(y-y_1)(y-y_2)}\right|\mathrm{d}y\\
\lesssim&\frac{1}{y_2-y_1+\delta}\int_{y_2+\delta}^{2y_2-y_1+\delta}\frac{1}{y-y_2}\mathrm{d}y+\int_{2y_2-y_1+\delta}^{\infty}\frac{1}{(y-y_2)^2}\mathrm{d}y\\
\lesssim&\frac{1+\ln\left(1+\frac{y_2-y_1}{\delta}\right)}{(y_2-y_1+\delta)}.
\end{aligned}
\end{equation}

Thanks to
\begin{equation}\nonumber
\begin{aligned}
\left|\frac{-2y}{(1-y^2-\lambda)^2}\right|
\lesssim&\frac{y}{(y-y_1)^2(y-y_2)^2}\lesssim\frac{1}{(y_2-y_1+\delta)(y-y_2)^2}
\end{aligned}
\end{equation}
for $y\in (y_2+\delta,1)$, we infer that
\begin{equation}\nonumber
\begin{aligned}
\int_{y_2+\delta}^1\left|\frac{-2y}{(1-y^2-\lambda)^2}\right|^2\mathrm{d}y\lesssim&\frac{1}{(y_2-y_1+\delta)^2}\int_{y_2+\delta}^1\frac{1}{(y-y_2)^4}\mathrm{d}y\\
\lesssim&\frac{1}{(y_2-y_1+\delta)^2\delta^3},
\end{aligned}
\end{equation}
which gives (\ref{eq:P-L2-w}).
\end{proof}

Similarly, we can show that for $0<\delta<\frac {y_2-y_1} 4$,
\begin{align}
&\left\Vert \frac{1}{1-y^2-\lambda}\right\Vert_{L^2\big((y_1+\delta,y_2-\delta)\big)}\lesssim \frac{1}{(y_2-y_1)\delta^{\frac{1}{2}}}.\label{eq:P-L2-in}
\end{align}

\subsection{Case of $\lambda>1$}

By integration by parts, we get
\begin{equation}\nonumber
\begin{aligned}
\Big|&\mathrm{Im}\big\langle -\nu(\partial_y^2-|k|^2) w+ik[(1-y^2-\lambda)w+2\varphi],w \big\rangle\Big|\\
&=|k|\left(\int_{-1}^1 (\lambda-1+y^2)|w|^2\mathrm{d}y+2\Vert \varphi'\Vert_{L^2}^2+2|k|^2\Vert \varphi\Vert_{L^2}^2\right),
\end{aligned}
\end{equation}
which gives
$$\int_{-1}^1 (\lambda-1+y^2)|w|^2\mathrm{d}y+2\Vert \varphi'\Vert_{L^2}^2+2|k|^2\Vert \varphi\Vert_{L^2}^2 \leq \frac{1}{|k|}\Vert F\Vert_{L^2}\Vert w \Vert_{L^2}.$$
Then we obtain
\begin{equation}\nonumber
\begin{aligned}
\Vert w \Vert_{L^2}^2 \leq &\Vert w\Vert_{L^2\big([-1,1]\setminus (-\delta,\delta)\big)}^2 +2\delta\Vert w \Vert_{L^\infty}^2\\
\lesssim&\frac{1}{|k|\delta^2}\Vert F\Vert_{L^2}\Vert w \Vert_{L^2}+\delta \Vert w \Vert_{L^2} \Vert w' \Vert_{L^2}\\
\lesssim&\frac{1}{|k|\delta^2}\Vert F\Vert_{L^2}\Vert w \Vert_{L^2}+\delta \nu^{-\frac{1}{2}}\Vert F \Vert_{L^2}^{\frac{1}{2}} \Vert w \Vert_{L^2}^{\frac{3}{2}},
\end{aligned}
\end{equation}
here we used $\lambda-1+y^2\geq \delta^2$ for $y\in [-1,1]\setminus (-\delta,\delta)$. This shows by taking $\delta=(\nu |k|^{-1})^{\frac{1}{4}}$ that
$$\Vert w \Vert_{L^2} \lesssim \nu^{-\frac{1}{2}}|k|^{-\frac{1}{2}}\Vert F\Vert_{L^2}.$$

\subsection{Case of $\lambda\in [0,1]$}

Let $-1 \leq y_1 \leq 0 \leq y_2 \leq 1$ so that $\lambda=1-y_1^2=1-y_2^2$.
We introduce the decomposition $\varphi=\varphi_1+\varphi_2$, where $\varphi_1$ solves
\begin{equation}\nonumber
\left\{
\begin{array}{lll}
(\partial_y^2-|k|^2)\varphi_1=w,\\
\varphi_1(\pm 1)=\varphi_1(y_1)=\varphi_1(y_2)=0,
\end{array}
\right.
\end{equation}
and $\varphi_2$ solves
\begin{equation}\nonumber
\left\{
\begin{array}{lll}
(\partial_y^2-|k|^2)\varphi_2=0,\\
\varphi_2(\pm 1)=0,\quad \varphi_2(y_1)=\varphi(y_1),\quad \varphi_2(y_2)=\varphi(y_2).
\end{array}
\right.
\end{equation}
It is easy to see that
\begin{equation}\nonumber
\varphi_2(y)=
\left\{
\begin{aligned}
&\frac{\sinh |k|(y+1)}{\sinh |k|(y_1+1)}\varphi (y_1), &y\in [-1,y_1],\\
&\frac{\sinh |k|(y-y_1)}{\sinh |k|(y_2-y_1)}\varphi(y_2)+\frac{\sinh |k|(y_2-y)}{\sinh |k|(y_2-y_1)}\varphi(y_1), &y\in [y_1,y_2],\\
&\frac{\sinh |k|(1-y)}{\sinh |k|(1-y_2)}\varphi (y_2),&y\in [y_2,1].
\end{aligned}
\right.
\end{equation}

We first give $L^2$ estimate of $w$ outside the interval $(y_1,y_2)$.

\begin{lemma}\label{lem:outside}
It holds that for any $\delta \in (0,1]$
\begin{equation}\nonumber
\Vert w \Vert_{L^2\big((-1,1)\setminus (y_1,y_2)\big)}^2 \leq C \mathcal{E}_1(w),
\end{equation}
where
\begin{equation}\label{def:E1}
\begin{aligned}
\mathcal{E}_1(w)=&\frac{\Vert F \Vert_{L^2}\Vert w \Vert_{L^2}}{|k|(y_2-y_1+\delta)\delta}+\frac{\nu\Vert w'\overline{w}\Vert_{L^\infty(B(y_1,\delta)\cup B(y_2,\delta))}}{|k|(y_2-y_1+\delta)\delta}\\
&+\frac{\nu \Vert w' \Vert_{L^2}\Vert w \Vert_{L^\infty}}{|k|(y_2-y_1+\delta)\delta^{\frac{3}{2}}}+\frac{\Vert \varphi_2 \Vert_{L^\infty}^2}{(y_2-y_1+\delta)^2\delta}+\delta \Vert w \Vert_{L^\infty}^2.
\end{aligned}
\end{equation}
\end{lemma}

\begin{proof}
By integration by parts, we get
\begin{equation}\nonumber
\begin{aligned}
&\Big|\mathrm{Im}\big\langle -\nu(\partial_y^2-|k|^2)w+ik[(1-y^2-\lambda)w+2\varphi], w \chi_{(-1,1)\setminus (y_1,y_2)}\big\rangle\Big|\\
&=\bigg|k\int_{(-1,1)\setminus (y_1,y_2)}(1-y^2-\lambda)|w|^2\mathrm{d}y+\nu \mathrm{Im}\big(w'\overline{w}(y_2)-w'\overline{w}(y_1)\big)\\
&\quad+\mathrm{Im}\int_{(-1,1)\setminus (y_1,y_2)}2ki\varphi \overline{w}\mathrm{d}y\bigg|\\
&\geq  |k|\int_{(-1,1)\setminus (y_1,y_2)}(\lambda-1+y^2)|w|^2\mathrm{d}y-2|k| \int_{(-1,1)\setminus (y_1,y_2)}\varphi_1\overline{w}\mathrm{d}y\\
&\quad-\nu\big(|w'\overline{w}(y_2)|+|w'\overline{w}(y_1)|\big)-2|k|\left| \int_{(-1,1)\setminus (y_1,y_2)}\varphi_2\overline{w}\mathrm{d}y\right|,
\end{aligned}
\end{equation}
which gives
\begin{equation}\label{eq:out-w}
\begin{aligned}
&\int_{(-1,1)\setminus (y_1,y_2)}(\lambda-1+y^2)|w|^2\mathrm{d}y+ \int_{(-1,1)\setminus (y_1,y_2)}-2\varphi_1\overline{w}\mathrm{d}y \\
&\leq |k|^{-1}\Vert F\Vert_{L^2}\Vert w \Vert_{L^2}\\
&\quad+\frac{\nu}{|k|}\Vert w'\overline{w}\Vert_{   L^\infty(B(y_1,\delta)\cup B(y_2,\delta))}+\left| \int_{(-1,1)\setminus (y_1,y_2)}2\varphi_2\overline{w}\mathrm{d}y\right|.
\end{aligned}
\end{equation}
Similarly, we have
\begin{equation}\nonumber
\begin{aligned}
&\bigg|\mathrm{Im} \bigg\langle -\nu(\partial_y^2-|k|^2)w\\
&\qquad+ik[(1-y^2-\lambda)w+2\varphi], \frac{w}{1-y^2-\lambda} \chi_{(-1,1)\setminus (y_1-\delta,y_2+\delta)}\bigg\rangle\bigg|\\
=&\bigg|\mathrm{Im}\bigg(-\nu\frac{w'\overline{w}}{(1-y^2-\lambda)}\Big|_{-1}^{y_1-\delta}-\nu\frac{w'\overline{w}}{(1-y^2-\lambda)}\Big|_{y_2+\delta}^{1}\bigg)\\
&+\mathrm{Im}\bigg(2ik\int_{(-1,1)\setminus (y_1-\delta,y_2+\delta)}\frac{\varphi \overline{w}}{1-y^2-\lambda}\mathrm{d}y\\
&+\nu\int_{(-1,1)\setminus (y_1-\delta,y_2+\delta)}\frac{2yw'\overline{w}}{(1-y^2-\lambda)^2}\mathrm{d}y\\
&+ik\int_{(-1,1)\setminus (y_1-\delta,y_2+\delta)}|w|^2\mathrm{d}y\bigg)\bigg|\\
\geq& -\nu\Vert w'\overline{w}\Vert_{L^\infty(B(y_1,\delta)\cup B(y_2,\delta))}\left(\frac{1}{|y_1^2-(y_1-\delta)^2|}+\frac{1}{|y_2^2-(y_2+\delta)^2|}\right)\\
&-\nu\Vert w \Vert_{L^\infty}\Vert w' \Vert_{L^2}\left\Vert \frac{2y}{(1-y^2-\lambda)^2}\right\Vert_{L^2((-1,1)\setminus (y_1-\delta,y_2+\delta))}\\
&+\frac{|k|}{2}\int_{(-1,1)\setminus (y_1-\delta,y_2+\delta)}|w|^2\mathrm{d}y\\
&-|k|\int_{(-1,1)\setminus (y_1-\delta,y_2+\delta)}\frac{2|\varphi|^2}{(1-y^2-\lambda)^2}\mathrm{d}y,
\end{aligned}
\end{equation}
which shows that
\begin{equation}\nonumber
\begin{aligned}
\frac{|k|}{2}\int_{(-1,1)\setminus (y_1-\delta,y_2+\delta)}|w|^2\mathrm{d}y
\leq& \Vert F\Vert_{L^2}\left\Vert \frac{w}{(1-y^2-\lambda)}\right\Vert_{L^2((-1,1)\setminus (y_1-\delta,y_2+\delta))}\\
&+\frac{2\nu \Vert w'\overline{w}\Vert_{L^\infty(B(y_1,\delta)\cup B(y_2,\delta))}}{|y_1^2-(y_1-\delta)^2|}\\
&+\nu\Vert w \Vert_{L^\infty}\Vert w' \Vert_{L^2}\left\Vert \frac{2y}{(1-y^2-\lambda)^2}\right\Vert_{L^2((-1,1)\setminus (y_1-\delta,y_2+\delta))}\\
&+|k|\int_{(-1,1)\setminus (y_1-\delta,y_2+\delta)}\frac{2|\varphi|^2}{(1-y^2-\lambda)^2}\mathrm{d}y.
\end{aligned}
\end{equation}
This along with Lemma \ref{lem:P} shows that
\begin{equation}\label{eq:out-dw}
\begin{aligned}
\Vert w \Vert_{L^2((-1,1)\setminus (y_1,y_2))}^2 \leq & \Vert w \Vert_{L^2((-1,1)\setminus (y_1-\delta,y_2+\delta))}^2+2\delta \Vert w \Vert_{L^\infty}^2\\
\lesssim & \frac{\Vert F \Vert_{L^2}\Vert w \Vert_{L^2}}{|k|(y_2-y_1+\delta)\delta}+\frac{\nu\Vert w'\overline{w}\Vert_{L^\infty(B(y_1,\delta)\cup B(y_2,\delta))}}{|k|(y_2-y_1+\delta)\delta}\\
&+\frac{\nu \Vert w'\Vert_{L^2}\Vert w \Vert_{L^\infty}}{|k|\delta^{\frac{3}{2}}(y_2-y_1+\delta)}+\delta \Vert w \Vert_{L^\infty}^2\\
&+\int_{(-1,1)\setminus (y_1-\delta,y_2+\delta)}\frac{|\varphi|^2}{(1-y^2-\lambda)^2}\mathrm{d}y.
\end{aligned}
\end{equation}

It remains to estimate $\int_{(-1,1)\setminus (y_1-\delta,y_2+\delta)}\frac{|\varphi|^2}{(1-y^2-\lambda)^2}\mathrm{d}y$, which is bounded as
\begin{equation}\nonumber
\begin{aligned}\int_{(-1,1)\setminus (y_1-\delta,y_2+\delta)}\frac{|\varphi|^2}{(1-y^2-\lambda)^2}\mathrm{d}y \lesssim &\int_{(-1,1)\setminus (y_1-\delta,y_2+\delta)}\frac{|\varphi_1|^2}{(1-y^2-\lambda)^2}\mathrm{d}y\\
&+\int_{(-1,1)\setminus (y_1-\delta,y_2+\delta)}\frac{|\varphi_2|^2}{(1-y^2-\lambda)^2}\mathrm{d}y.
\end{aligned}
\end{equation}
Thanks to $|1-y^2-\lambda|=(y-y_1)(y-y_2),$ we get by Hardy's inequality that
\begin{equation}\label{eq:out-phi1}
\begin{aligned}
&\int_{(-1,1)\setminus (y_1-\delta,y_2+\delta)}\frac{|\varphi_1|^2}{(1-y^2-\lambda)^2}\mathrm{d}y \\
&\lesssim \int_{-1}^{y_1-\delta}\frac{|\varphi_1|^2}{(1-y^2-\lambda)^2}\mathrm{d}y+\int_{y_2+\delta}^{1}\frac{|\varphi_1|^2}{(1-y^2-\lambda)^2}\mathrm{d}y\\
&\leq  \frac{1}{(y_2-y_1+\delta)^2}\bigg(\int_{y_2+\delta}^{1}\frac{|\int_{y_1}^{y}\varphi_1'\mathrm{d}z|^2}{(y-y_1)^2}\mathrm{d}y+ \int_{y_2+\delta}^{1}\frac{|\int_{y_2}^{y}\varphi_1'\mathrm{d}z|^2}{(y-y_2)^2}\mathrm{d}y\bigg)\\
&\lesssim \frac{1}{(y_2-y_1+\delta)^2} \int_{(-1,1)\setminus (y_1-\delta,y_2+\delta)}|\varphi_1'|^2\mathrm{d}y.
\end{aligned}
\end{equation}
By (\ref{eq:energy-cut2}) and (\ref{eq:out-w}), we have
\begin{equation}\nonumber
\begin{aligned}
&\frac{1}{(y_2-y_1+\delta)^2} \int_{(-1,1)\setminus (y_1,y_2)}|\varphi_1'|^2\mathrm{d}y\\
&\leq \frac{1}{(y_2-y_1+\delta)^2}\big\langle -\varphi_1,w\chi_{([-1,1]\setminus (y_1,y_2))}\big\rangle\\
&\lesssim\frac{1}{(y_2-y_1+\delta)^2}\bigg(|k|^{-1}\Vert F\Vert_{L^2}\Vert w \Vert_{L^2}+\frac{\nu}{|k|}\Vert w'\overline{w}\Vert_{   L^\infty(B(y_1,\delta)\cup B(y_2,\delta))}\\
&\qquad+\left| \int_{(-1,1)\setminus (y_1,y_2)}2\varphi_2\overline{w}\mathrm{d}y\right|\bigg)\\
&\lesssim \frac{\Vert F\Vert_{L^2}\Vert w \Vert_{L^2}}{|k|(y_2-y_1+\delta)^2}+\frac{\nu \Vert w'\overline{w}\Vert_{  L^\infty(B(y_1,\delta)\cup B(y_2,\delta))}}{|k|(y_2-y_1+\delta)^2}\\
&\quad +\frac{\Vert \varphi_2\Vert_{L^\infty}}{\delta^{\frac{1}{2}}(y_2-y_1+\delta)}\frac{\delta^{\frac{3}{2}}\Vert w \Vert_{L^\infty}+\delta^{\frac{1}{2}}\Vert w \Vert_{L^1((-1,1)\setminus (y_1-\delta,y_2+\delta))}}{(y_2-y_1+\delta)}.
\end{aligned}
\end{equation}
By Lemma \ref{lem:P}, we have
\begin{align*}
&\Vert w \Vert_{L^1((-1,1)\setminus (y_1-\delta,y_2+\delta))}^2 \\
&\leq \left\Vert \frac{1}{1-y^2-\lambda}\right\Vert_{L^1((-1,1)\setminus (y_1-\delta,y_2+\delta))}\left(\int_{(-1,1)\setminus (y_1-\delta,y_2+\delta)}(\lambda-1+y^2)|w|^2\mathrm{d}y\right)\\
&\lesssim\frac{1+\ln (1+\frac{y_2-y_1}{\delta})}{(y_2-y_1+\delta)} \bigg(\frac{\Vert F\Vert_{L^2}\Vert w \Vert_{L^2}}{|k|}+\frac{\nu}{|k|}\Vert w'\overline{w}\Vert_{   L^\infty(B(y_1,\delta)\cup B(y_2,\delta))}\\
&\qquad +\left| \int_{(-1,1)\setminus (y_1,y_2)}2\varphi_2\overline{w}\mathrm{d}y\right|\bigg)\\
&\leq \frac{1+\ln (1+\frac{y_2-y_1}{\delta})}{(y_2-y_1+\delta)}\big(\mathcal{E}_1(w)(y_2-y_1+\delta)\delta+\Vert \varphi_2 \Vert_{L^\infty}(2\delta \Vert w \Vert_{L^\infty}\\
&\qquad +\Vert w \Vert_{L^1((-1,1)\setminus (y_1-\delta,y_2+\delta))})\bigg)\\
&\lesssim \mathcal{E}_1(w)\delta\left(1+\ln (1+\frac{y_2-y_1}{\delta})\right)\\
&\qquad +\frac{1+\ln (1+\frac{y_2-y_1}{\delta})}{(y_2-y_1+\delta)}\Vert \varphi_2\Vert_{L^\infty}\Vert w \Vert_{L^1((-1,1)\setminus (y_1-\delta,y_2+\delta))}),
\end{align*}
which implies
\begin{equation}\label{eq:out-w-L1}
\begin{aligned}
\Vert w \Vert_{L^1((-1,1)\setminus (y_1-\delta,y_2+\delta))}^2
&\lesssim \mathcal{E}_1(w)\delta\left(1+\ln (1+\frac{y_2-y_1}{\delta})\right)\\
&\qquad+\left|\frac{1+\ln (1+\frac{y_2-y_1}{\delta})}{(y_2-y_1+\delta)}\right|^2\Vert \varphi_2\Vert_{L^\infty}^2\\
&\lesssim \mathcal{E}_1(w)\delta(1+\frac{y_2-y_1}{\delta})= \mathcal{E}_1(w)(y_2-y_1+\delta).
\end{aligned}
\end{equation}
Then we conclude from (\ref{eq:out-phi1}) that
\begin{equation}\label{eq:out-phi1-2}
\begin{aligned}
\int_{(-1,1)\setminus (y_1-\delta,y_2+\delta)}\frac{|\varphi_1|^2}{(1-y^2-\lambda)^2}\mathrm{d}y \lesssim \mathcal{E}_1(w).
\end{aligned}
\end{equation}
On the other hand, we have
\begin{equation}\label{eq:out-phi2}
\begin{aligned}
&\int_{(-1,1)\setminus (y_1-\delta,y_2+\delta)}\frac{|\varphi_2|^2}{(1-y^2-\lambda)^2}\mathrm{d}y \\
&\lesssim \Vert \varphi_2\Vert_{L^\infty}^2\int_{(-1,1)\setminus (y_1-\delta,y_2+\delta)}\frac{1}{\left((y-y_1)(y-y_2)\right)^2}\mathrm{d}y\\
&\lesssim \frac{\Vert \varphi_2\Vert_{L^\infty}^2}{(y_2-y_1+\delta)^2\delta} \leq \mathcal{E}_1(w).
\end{aligned}
\end{equation}

Finally, the lemma follows from (\ref{eq:out-dw}), (\ref{eq:out-phi1-2}) and (\ref{eq:out-phi2}).
\end{proof}

Next we give $L^2$ estimate of $w$ in the interval $(y_1,y_2)$.

\begin{lemma}\label{lem:inside}
Let $\delta \in (0,\frac{y_2-y_1}{4}]$. It holds that
\begin{equation}\nonumber
\Vert w \Vert_{L^2(y_1,y_2)}^2 \leq C \mathcal{E}_2(w),
\end{equation}
where
\begin{equation}\label{def:E2}
\begin{aligned}
\mathcal{E}_2(w)=&\frac{\Vert F \Vert_{L^2}\Vert w \Vert_{L^2}}{|k|\delta(y_2-y_1)}+\frac{\nu \Vert w'\overline{w}\Vert_{L^\infty(B(y_1,\delta)\cup B(y_2,\delta))}}{|k|\delta (y_2-y_1)}\\
&+\delta \Vert w \Vert_{L^\infty}^2+\frac{\nu^2\Vert w \Vert_{L^\infty}^2}{|k|^2(y_2-y_1)^3\delta^4}+\frac{\nu \Vert w'\Vert_{L^2}\Vert w \Vert_{L^\infty}}{|k|\delta^{\frac{3}{2}}(y_2-y_1)}\\
&+\frac{\Vert \varphi_2\Vert_{L^\infty}^2}{(y_2-y_1)^2\delta}+\delta^3\Vert w'\Vert_{L^\infty(B(y_1,\delta)\cup B(y_2,\delta))}^2+\frac{\Vert F \Vert_{L^2}^2}{|k|^2(y_2-y_1)^3\delta}\\
&+\frac{\nu^2}{|k|^2}\left(\frac{\Vert w'\Vert_{L^\infty(B(y_1,\delta)\cup B(y_2,\delta))}^2}{(y_2-y_1)^3\delta^2}+\frac{\Vert w'\Vert_{L^2}^2}{(y_2-y_1)^3\delta^3}\right).
\end{aligned}
\end{equation}
\end{lemma}
\begin{proof}
By integration by parts, we get
\begin{equation}\nonumber
\begin{aligned}
&\big|\mathrm{Im} \big\langle -\nu(\partial_y^2-|k|^2)w+ik[(1-y^2-\lambda)w+2\varphi], w \chi_{(y_1,y_2)}\big\rangle\big|\\
&=\bigg|k\int_{y_1}^{y_2}(1-y^2-\lambda)|w|^2\mathrm{d}y-\nu \mathrm{Im}\big(w'\overline{w}(y_2)-w'\overline{w}(y_1)\big)+k\int_{y_1}^{y_2}2\varphi \overline{w}\mathrm{d}y\bigg|\\
&\geq |k|\left(\int_{y_1}^{y_2}(1-y^2-\lambda)|w|^2\mathrm{d}y+2 \int_{y_1}^{y_2}\varphi_1\overline{w}\mathrm{d}y\right)\\
&\quad-\nu\big(|w'\overline{w}(y_2)|+|w'\overline{w}(y_1)|\big)-2|k|\left| \int_{y_1}^{y_2}\varphi_2\overline{w}\mathrm{d}y\right|,
\end{aligned}
\end{equation}
which gives
\begin{equation}\label{eq:in-w}
\begin{aligned}
&\frac{1}{(y_2-y_1)^2}\bigg(\int_{y_1}^{y_2}(1-y^2-\lambda)|w|^2\mathrm{d}y+\int_{y_1}^{y_2}2\varphi_1\overline{w}\mathrm{d}y\bigg)\\
&\leq \frac{1}{(y_2-y_1)^2}\bigg(\frac{1}{|k|}\Vert F\Vert_{L^2}\Vert w \Vert_{L^2}+\frac{2\nu}{|k|}\Vert w'\overline{w}\Vert_{L^\infty(B(y_1,\delta)\cup B(y_2,\delta))}\\
&\quad+2\left| \int_{y_1}^{y_2}\varphi_2\overline{w}\mathrm{d}y\right|\bigg)\\
&\lesssim \mathcal{E}_2(w)+\frac{2\Vert \varphi_2\Vert_{L^\infty}}{(y_2-y_1)\delta^{\frac{1}{2}}}\delta^{\frac{1}{2}}\Vert w \Vert_{L^\infty}\lesssim\mathcal{E}_2(w).
\end{aligned}
\end{equation}
Similarly, we have
\begin{equation}\nonumber
\begin{aligned}
&\bigg|\mathrm{Im} \bigg\langle -\nu(\partial_y^2-|k|^2)w\\
&\quad \quad +ik[(1-y^2-\lambda)w+2\varphi], \frac{w}{1-y^2-\lambda} \chi_{(y_1+\delta,y_2-\delta)}\bigg\rangle\bigg|\\
=&\bigg|\mathrm{Im}\bigg(-\nu\frac{w'\overline{w}}{(1-y^2-\lambda)}\Big|_{y_1+\delta}^{y_2-\delta}+
+2ik\int_{y_1+\delta}^{y_2-\delta}\frac{\varphi \overline{w}}{1-y^2-\lambda}\mathrm{d}y\\
&+\nu\int_{y_1+\delta}^{y_2-\delta}\frac{2yw'\overline{w}}{(1-y^2-\lambda)^2}\mathrm{d}y+ik\int_{y_1+\delta}^{y_2-\delta}|w|^2\mathrm{d}y\bigg)\bigg|\\
\geq &-\nu\Vert w'\overline{w}\Vert_{L^\infty(B(y_1,\delta)\cup B(y_2,\delta))}\left(\frac{1}{|y_1^2-(y_1+\delta)^2|}+\frac{1}{|y_2^2-(y_2-\delta)^2|}\right)\\
&-\nu\Vert w \Vert_{L^\infty}\Vert w' \Vert_{L^2}\left\Vert \frac{2y}{(1-y^2-\lambda)^2}\right\Vert_{L^2(y_1+\delta,y_2-\delta)}\\
&+\frac{|k|}{2}\int_{y_1+\delta}^{y_2-\delta}|w|^2\mathrm{d}y-|k|\int_{y_1+\delta}^{y_2-\delta}\frac{2|\varphi|^2}{(1-y^2-\lambda)^2}\mathrm{d}y,
\end{aligned}
\end{equation}
which gives
\begin{align*}
\frac{|k|}{2}\int_{y_1+\delta}^{y_2-\delta}|w|^2\mathrm{d}y\leq &\Vert F\Vert_{L^2}\left\Vert \frac{w}{(1-y^2-\lambda)}\right\Vert_{L^2(y_1+\delta,y_2-\delta)}\\
&+\frac{2\nu \Vert w'\overline{w}\Vert_{L^\infty(B(y_1,\delta)\cup B(y_2,\delta))}}{|y_1^2-(y_1-\delta)^2|}\\
&+\nu\Vert w \Vert_{L^\infty}\Vert w' \Vert_{L^2}\left\Vert \frac{2y}{(1-y^2-\lambda)^2}\right\Vert_{L^2(y_1+\delta,y_2-\delta)}\\
&+|k|\int_{y_1+\delta}^{y_2-\delta}\frac{2|\varphi|^2}{(1-y^2-\lambda)^2}\mathrm{d}y,
\end{align*}
which along with $y_2-y_1 \geq 4 \delta$ gives
\begin{equation}\label{eq:in-dw}
\begin{aligned}
\Vert w \Vert_{L^2(y_1,y_2)}^2 \lesssim &\frac{1}{|k|\delta(y_2-y_1)}\Vert F\Vert_{L^2}\Vert w \Vert_{L^2}+\frac{\nu\Vert w'\overline{w}\Vert_{L^\infty(B(y_1,\delta)\cup B(y_2,\delta))}}{|k|\delta (y_2-y_1)}\\
&+\frac{\nu\Vert w'\Vert_{L^2}\Vert w \Vert_{L^\infty}}{|k|\delta^{\frac{3}{2}}(y_2-y_1)}+\delta \Vert w \Vert_{L^\infty}^2+\int_{y_1+\delta}^{y_2-\delta}\frac{2|\varphi|^2}{(1-y^2-\lambda)^2}\mathrm{d}y.
\end{aligned}
\end{equation}

Next we estimate the last term  as follows
$$\int_{y_1+\delta}^{y_2-\delta}\frac{2|\varphi|^2}{(1-y^2-\lambda)^2}\mathrm{d}y \le
 \int_{y_1+\delta}^{y_2-\delta}\frac{2|\varphi_1|^2}{(1-y^2-\lambda)^2}\mathrm{d}y+\int_{y_1+\delta}^{y_2-\delta}\frac{2|\varphi_2|^2}{(1-y^2-\lambda)^2}\mathrm{d}y.$$
First of all, we have
\begin{equation}\label{eq:in-phi2}
\begin{aligned}
\int_{y_1+\delta}^{y_2-\delta}\frac{2|\varphi_2|^2}{(1-y^2-\lambda)^2}\mathrm{d}y \lesssim &\Vert \varphi_2\Vert_{L^\infty}^2 \int_{y_1+\delta}^{y_2-\delta}\frac{2}{(1-y^2-\lambda)^2}\mathrm{d}y\\
\lesssim&\frac{\Vert \varphi_2\Vert_{L^\infty}^2 }{(y_2-y_1)^2\delta} \lesssim \mathcal{E}_2(w).
\end{aligned}
\end{equation}
It remains to prove that
\ben\label{eq:in-phi1}
\int_{y_1+\delta}^{y_2-\delta}\frac{2|\varphi_1|^2}{(1-y^2-\lambda)^2}\mathrm{d}y \lesssim \mathcal{E}_2(w).
\een
Then the lemma follows from \eqref{eq:in-dw}, \eqref{eq:in-phi1} and \eqref{eq:in-phi2}.\smallskip

Thanks to Lemma \ref{lem:hardy}, (\ref{eq:energy-key}) and (\ref{eq:in-w}), we get
\begin{equation}\label{eq:in-phi1-est1}
\begin{aligned}
&\int_{y_1+\delta}^{y_2-\delta}\frac{2|\varphi_1|^2}{(1-y^2-\lambda)^2}\mathrm{d}y \\ &\lesssim\frac{2}{(y_2-y_1)^2}\int_{y_1}^{y_2}(1-y^2-\lambda)^2\left|\left(\frac{\varphi_1}{(1-y^2-\lambda)}\right)'\right|^2\mathrm{d}y\\
&\qquad+\frac{2|\varphi_1(0)|^2}{(y_2-y_1)^3}\\
&\lesssim \frac{1}{(y_2-y_1)^2}\left( \int_{y_1}^{y_2} (1-y^2-\lambda)|w|^2\mathrm{d}y+\langle 2 \varphi_1,w \chi_{(y_1,y_2)}\rangle\right)\\
&\qquad+\frac{2|\varphi_1(0)|^2}{(y_2-y_1)^3}\\
&\lesssim \mathcal{E}_2(w)+\frac{2|\varphi_1(0)|^2}{(y_2-y_1)^3}.
\end{aligned}
\end{equation}

To control $\frac{2|\varphi_1(0)|^2}{(y_2-y_1)^3}$, we notice that for any $\delta \leq \theta \leq \frac{y_2-y_1}{4}$,
\begin{equation}\nonumber
\begin{aligned}
\int_0^{y_2-\theta}&\left(\frac{\varphi_1(y)}{1-y^2-\lambda}\right)'(y_2-\theta-y)\mathrm{d}y\\
&-\int_{y_1+\theta}^{0}\left(\frac{\varphi_1(y)}{1-y^2-\lambda}\right)'(y-y_1-\theta)\mathrm{d}y\\
=&\int_{y_1+\theta}^{y_2-\theta}\frac{\varphi_1(y)}{1-y^2-\lambda}\mathrm{d}y-\frac{\varphi_1(0)}{1-\lambda}(y_2-y_1-2\theta),
\end{aligned}
\end{equation}
which gives
\begin{align*}
&\frac{\varphi_1(0)}{1-\lambda}(y_2-y_1-2\theta)\leq \left|\int_{y_1+\theta}^{y_2-\theta}\frac{\varphi_1(y)}{1-y^2-\lambda}\mathrm{d}y\right|\\
&\qquad+\int_{y_1+\theta}^{y_2-\theta}\left|\left(\frac{\varphi_1(y)}{1-y^2-\lambda}\right)'\right|\min(y_2-y,y-y_1)\mathrm{d}y\\
&\lesssim \left|\int_{y_1+\theta}^{y_2-\theta}\frac{\varphi_1(y)}{1-y^2-\lambda}\mathrm{d}y\right|\\
&\qquad+(y_2-y_1)^{\frac{1}{2}}\left(\int_{y_1+\theta}^{y_2-\theta}\left|\left(\frac{\varphi_1(y)}{1-y^2-\lambda}\right)'\right|^2\min(y_2-y,y-y_1)^2\mathrm{d}y\right)^{\frac{1}{2}}\\
&\lesssim \left|\int_{y_1+\theta}^{y_2-\theta}\frac{\varphi_1(y)}{1-y^2-\lambda}\mathrm{d}y\right|\\
&\qquad+(y_2-y_1)^{\frac{1}{2}}\left(\int_{y_1+\theta}^{y_2-\theta}\left|\left(\frac{\varphi_1(y)}{1-y^2-\lambda}\right)'\right|^2\frac{(1-y^2-\lambda)^2}{(y_2-y_1)^2}\mathrm{d}y\right)^{\frac{1}{2}}.
\end{align*}
Thanks to $1-\lambda\thicksim(y_2-y_1)^2$, we obtain
\begin{equation}\label{eq:in-phi1-b}
\begin{aligned}
\frac{|\varphi_1(0)|^2}{(y_2-y_1)^3}\lesssim &\frac{\left|\int_{y_1+\theta}^{y_2-\theta}\frac{\varphi_1(y)}{1-y^2-\lambda}\mathrm{d}y\right|^2}{y_2-y_1}
\\
&+\int_{y_1+\theta}^{y_2-\theta}\left|\left(\frac{\varphi_1(y)}{1-y^2-\lambda}\right)'\right|^2\frac{(1-y^2-\lambda)^2}{(y_2-y_1)^2}\mathrm{d}y.
\end{aligned}
\end{equation}
By (\ref{eq:energy-key}) and (\ref{eq:in-w}), we have
\ben\label{eq:in-phi1-b0}
\int_{y_1+\theta}^{y_2-\theta}\left|\left(\frac{\varphi_1(y)}{1-y^2-\lambda}\right)'\right|^2\frac{(1-y^2-\lambda)^2}{(y_2-y_1)^2}\mathrm{d}y\lesssim \mathcal{E}_2(w).
\een

To control $\frac1 {y_2-y_1}{\left|\int_{y_1+\theta}^{y_2-\theta}\frac{\varphi_1(y)}{1-y^2-\lambda}\mathrm{d}y\right|^2}$, we consider two cases.\smallskip

\noindent{\bf Case 1.} $|k|^2(y_2-y_1)^2 \geq 1$. In this case, we take $\theta=\frac{y_2-y_1}{4}$ and obtain
\begin{equation}\nonumber
\begin{aligned}
\frac{1}{y_2-y_1}\left|\int_{y_1+\theta}^{y_2-\theta}\frac{\varphi_1(y)}{1-y^2-\lambda}\mathrm{d}y\right|^2 \leq &\frac{1}{y_2-y_1}\int_{y_1+\theta}^{y_2-\theta}|\varphi_1|^2\mathrm{d}y\int_{y_1+\theta}^{y_2-\theta}\frac{1}{(1-y^2-\lambda)^2}\mathrm{d}y\\
\lesssim&\frac{1}{(y_2-y_1)^3\theta}\int_{y_1+\theta}^{y_2-\theta}|\varphi_1|^2\mathrm{d}y\\
\lesssim &\frac{|k|^2}{(y_2-y_1)^2}\int_{y_1+\theta}^{y_2-\theta}|\varphi_1|^2\mathrm{d}y,
\end{aligned}
\end{equation}
which along with (\ref{eq:energy-key}) and (\ref{eq:in-w}) gives
\begin{equation}\label{eq:in-phi1-b1}
\frac{1}{y_2-y_1}\left|\int_{y_1+\theta}^{y_2-\theta}\frac{\varphi_1(y)}{1-y^2-\lambda}\mathrm{d}y\right|^2 \lesssim \mathcal{E}_2(w).
\end{equation}

\noindent{\bf Case 2.} $|k|^2(y_2-y_1)^2 \leq 1$.  In this case, we take $\theta=\delta$ and introduce
\begin{equation}\nonumber
\begin{aligned}
\chi(y)=\eta\left(\frac{y}{y_2-y_1}\right) \ \mathrm{with} \ \ \eta(z)=\left \{
\begin{array}{lll}
1,&|z| \leq 1,\\
0,&|z| \geq 2.
\end{array}
\right.
\end{aligned}
\end{equation}
We get by integration by parts that
$$\left|\int_{-1}^{1}\varphi'\chi'+|k|^2\varphi \chi\mathrm{d}y\right|=\left|\int_{-1}^{1}-w(y)\chi(y)\mathrm{d}y\right|.$$
Due to $\chi(y)=1$ for $y \in [y_1+\theta,y_2-\theta]$, we have
\begin{equation}\nonumber
\begin{aligned}
\left|\int_{y_1+\delta}^{y_2-\delta}w(y)\mathrm{d}y\right| \leq &\left|\int_{-1}^{1}\varphi'\chi'+|k|^2\varphi \chi\mathrm{d}y\right|+4\delta \Vert w\Vert_{L^\infty}\\
&+\int_{(-1,1)\setminus (y_1-\delta,y_2+\delta)}|w(y)|\mathrm{d}y.
\end{aligned}
\end{equation}
Recall that $F=-\nu(\partial_y^2-|k|^2)w+ik\big[(1-y^2-\lambda)w+2\varphi\big]$, which gives $w=\frac{-i\frac{1}{k}\big(F+\nu(\partial_y^2-|k|^2)w\big)-2\varphi}{(1-y^2-\lambda)}$ and then
\begin{equation}\nonumber
\begin{aligned}
&\left|\int_{y_1+\delta}^{y_2-\delta}\frac{2\varphi}{1-y^2-\lambda}\mathrm{d}y\right|-\left|\int_{y_1+\delta}^{y_2-\delta}\frac{i\frac{1}{k}(F+\nu(\partial_y^2-|k|^2)w)}{(1-y^2-\lambda)}\mathrm{d}y\right|\\
&\leq \left|\int_{-1}^{1}\varphi'\chi'+|k|^2\varphi \chi\mathrm{d}y\right|+4\delta \Vert w\Vert_{L^\infty}
+\int_{(-1,1)\setminus (y_1-\delta,y_2+\delta)}|w(y)|\mathrm{d}y,
\end{aligned}
\end{equation}
which yields that
\begin{equation}\nonumber
\begin{aligned}
&\frac{1}{y_2-y_1}\left|\int_{y_1+\delta}^{y_2-\delta}\frac{2\varphi_1}{1-y^2-\lambda}\mathrm{d}y\right|^2 \\
&\lesssim \frac{\delta^2}{y_2-y_1}\Vert w \Vert_{L^\infty}^2+\frac{1}{|k|^2(y_2-y_1)}\left|\int_{y_1+\delta}^{y_2-\delta} \frac{F}{1-y^2-\lambda}\mathrm{d}y\right|^2\\
& \quad+\frac{\nu^2}{|k|^2(y_2-y_1)}\left|\int_{y_1+\delta}^{y_2-\delta}\frac{w''}{1-y^2-\lambda}\mathrm{d}y\right|^2\\
&\quad+\frac{1}{y_2-y_1}\left|\int_{y_1+\delta}^{y_2-\delta}\frac{2\varphi_2}{1-y^2-\lambda}\mathrm{d}y\right|^2 +\frac{1}{y_2-y_1}\left|\int_{-1}^{1}\varphi'\chi'+|k|^2\varphi \chi\mathrm{d}y\right|^2\\
&\quad+\frac{1}{y_2-y_1}\left|\int_{(-1,1)\setminus (y_1-\delta,y_2+\delta)}|w(y)|\mathrm{d}y\right|^2
+\frac{\nu^2|k|^2}{y_2-y_1}\left|\int_{y_1+\delta}^{y_2-\delta}\frac{w}{1-y^2-\lambda}\mathrm{d}y\right|^2\\
&:=I_1+\cdots I_7.
\end{aligned}
\end{equation}

Thanks to $y_2-y_1 \geq 4\delta$, we have
$$I_1 \leq \delta \Vert w \Vert_{L^\infty}^2 \lesssim \mathcal{E}_2(w).$$
By \eqref{eq:P-L2-in}, we get
\begin{align*}
&I_2 \lesssim \frac{\Vert F\Vert_{L^2}^2}{|k|^2(y_2-y_1)}\left\Vert \frac{1}{1-y^2-\lambda}\right\Vert_{L^2(y_1+\delta,y_2-\delta)}^2\lesssim \frac{\Vert F\Vert_{L^2}^2}{|k|^2(y_2-y_1)^3\delta}\leq \mathcal{E}_2(w),\\
&I_4 \lesssim \Vert \varphi_2 \Vert_{L^\infty}^2 \left\Vert \frac{1}{1-y^2-\lambda}\right\Vert_{L^2(y_1+\delta,y_2-\delta)}^2 \lesssim\frac{\Vert \varphi_2\Vert_{L^\infty}^2}{(y_2-y_1)^2\delta} \leq \mathcal{E}_2(w).
\end{align*}
We get by integration by parts that
\begin{equation}\nonumber
\begin{aligned}
I_3=&\frac{\nu^2}{|k|^2(y_2-y_1)}\left|\frac{w'}{1-y^2-\lambda}\bigg|_{y_1+\delta}^{y_2-\delta}+\int_{y_1+\delta}^{y_2-\delta}\frac{w'(-2y)}{(1-y^2-\lambda)^2}\mathrm{d}y\right|^2\\
 \lesssim &\frac{\nu^2}{|k|^2}\left(\frac{\Vert w'\Vert_{L^\infty(B(y_1,\delta)\cup B(y_2,\delta))}^2}{\delta^2(y_2-y_1)^3}+\frac{\Vert w'\Vert_{L^2}^2}{\delta^3(y_2-y_1)^3}\right) \leq \mathcal{E}_2(w).
\end{aligned}
\end{equation}
By (\ref{eq:out-w-L1}), we have
\begin{equation}\nonumber
\begin{aligned}
I_6=&\frac{1}{y_2-y_1}\Vert w \Vert_{L^1((-1,1)\setminus (y_1-\delta,y_2+\delta))}^2 \\
\leq &\frac{C\mathcal{E}_1(w)(y_2-y_1+\delta)}{y_2-y_1} \lesssim \mathcal{E}_1(w) \lesssim\mathcal{E}_2(w).
\end{aligned}
\end{equation}

For $I_5$, we first notice  that
\begin{align*}
I_5 \lesssim &\frac{1}{y_2-y_1}\bigg(\Vert \varphi_1'\Vert_{L^2((-1,1)\setminus (y_1-\delta,y_2+\delta))}^2\Vert \chi'\Vert_{L^2}^2\\
&+\Vert \varphi_2'\Vert_{L^1}^2\Vert \chi'\Vert_{L^\infty}^2+|k|^4\Vert \varphi \Vert_{L^1(B(0,2(y_2-y_1))}^2\bigg)\\
\lesssim&\frac{1}{y_2-y_1}\bigg(\frac{\Vert \varphi_1'\Vert_{L^2((-1,1)\setminus (y_1-\delta,y_2+\delta))}^2}{(y_2-y_1)}\\
&\quad+\frac{\Vert \varphi_2'\Vert_{L^1}^2}{(y_2-y_1)^2}+|k|^4\left(\int_{-2(y_2-y_1)}^{2(y_2-y_1)}|\varphi|\mathrm{d}y\right)^2\bigg)\\
:=&I_5^1+I_5^2+I_5^3.
\end{align*}
By (\ref{eq:energy-cut2}), (\ref{eq:out-w}) and (\ref{eq:out-w-L1}), we have
\begin{equation}\nonumber
\begin{aligned}
I_5^1 \lesssim &\frac{1}{(y_2-y_1)^2}\Vert \varphi_1'\Vert_{L^2((-1,1)\setminus (y_1,y_2))}^2
\lesssim \frac{1}{2(y_2-y_1)^2}\big\langle -2\varphi_1,w\chi_{(-1,1)\setminus (y_1,y_2)}\big\rangle\\
\lesssim&\frac{1}{(y_2-y_1)^2}\bigg(\frac{\Vert F\Vert_{L^2}\Vert w \Vert_{L^2}}{|k|}+\frac{\nu}{|k|}\Vert w'\overline{w}\Vert_{L^\infty(B(y_1,\delta)\cup B(y_2,\delta))}\\
&\qquad+\left| \int_{(-1,1)\setminus (y_1,y_2)}2\varphi_2\overline{w}\mathrm{d}y\right|\bigg)\\
\lesssim&\mathcal{E}_2(w)+\frac{\Vert \varphi_2\Vert_{L^\infty}}{(y_2-y_1)^2}\big(\delta\Vert w\Vert_{L^\infty}+\Vert w \Vert_{L^1((-1,1)\setminus (y_1-\delta,y_2+\delta))}\big)\\
\lesssim&\mathcal{E}_2(w).
\end{aligned}
\end{equation}
Thanks to the definition of $\varphi_2$ and the monotonicity of $\sinh$, we have
$$I_5^2=\frac{\Vert\varphi_2'\Vert_{L^1}^2}{(y_2-y_1)^3}\lesssim\frac{\Vert\varphi_2\Vert_{L^\infty}^2}{(y_2-y_1)^3} \lesssim \frac{\Vert\varphi_2\Vert_{L^\infty}^2}{(y_2-y_1)^2\delta} \lesssim \mathcal{E}_2(w).$$
Due to $|k|^2(y_2-y_1)^2 \leq 1$, we get by \eqref{eq:energy-key}-(\ref{eq:in-w}) and \eqref{eq:energy-cut2}-\eqref{eq:out-w} that
\begin{equation}\nonumber
\begin{aligned}
I_5^3 \lesssim &\frac{|k|^4\left(\int_{-2(y_2-y_1)}^{2(y_2-y_1)}|\varphi_1|\mathrm{d}y\right)^2+|k|^4\left(\int_{-2(y_2-y_1)}^{2(y_2-y_1)}|\varphi_2|\mathrm{d}y\right)^2}{y_2-y_1}\\
\lesssim&|k|^4 \int_{-2(y_2-y_1)}^{2(y_2-y_1)}|\varphi_1|^2\mathrm{d}y+|k|^4(y_2-y_1)\Vert \varphi_2\Vert_{L^\infty}^2\\
\lesssim&\frac{1}{2(y_2-y_1)^2}2|k|^2\int_{-2(y_2-y_1)}^{2(y_2-y_1)}|\varphi_1|^2\mathrm{d}y+\frac{\Vert\varphi_2\Vert_{L^\infty}^2}{(y_2-y_1)^2\delta}\\
\lesssim&\mathcal{E}_2(w).
\end{aligned}
\end{equation}
This shows that
$$I_5 \lesssim\mathcal{E}_2(w).$$

For $I_7$, we have
\begin{equation}\nonumber
\begin{aligned}
I_7 \lesssim &\frac{\nu^2 |k|^2(y_2-y_1-2\delta)}{y_2-y_1}\int_{y_1+\delta}^{y_2-\delta} \frac{|w|^2}{(1-y^2-\lambda)^2}\mathrm{d}y\\
\lesssim &\frac{\nu^2}{|k|^2(y_2-y_1)^3\delta}\int_{y_1+\delta}^{y_2-\delta} \frac{|w|^2}{(1-y^2-\lambda)^2}\mathrm{d}y.
\end{aligned}
\end{equation}
On the other hand, we get by Hardy's inequality that
\begin{equation}\nonumber
\begin{aligned}
&\int_{y_1+\delta}^{y_2-\delta} \frac{|w|^2}{(1-y^2-\lambda)^2}\mathrm{d}y \\
&\lesssim \frac{1}{(y_2-y_1-\delta)\delta}\bigg(\int_{y_1+\delta}^0 \frac{\left|\int_{y_1}^y w'\mathrm{d}z\right|^2}{(y-y_1)^2}+\frac{\left|w(y_1)\right|^2}{(y-y_1)^2}\mathrm{d}y\\
&\qquad+\int_{0}^{y_2-\delta} \frac{\left|\int_{y_2}^y w'\mathrm{d}z\right|^2}{(y-y_2)^2}+\frac{\left|w(y_2)\right|^2}{(y-y_2)^2}\mathrm{d}y\bigg)\\
&\lesssim \frac{1}{\delta^2}\bigg(\int_{y_1+\delta}^0 |w'|^2+\frac{\left|w(y_1)\right|^2}{(y-y_1)^2}\mathrm{d}y+\int_{0}^{y_2-\delta} |w'|^2+\frac{\left|w(y_2)\right|^2}{(y-y_2)^2}\mathrm{d}y\bigg)\\
&\lesssim \frac{1}{\delta^2}\Vert w' \Vert_{L^2}^2+\frac{1}{\delta^3}\Vert w \Vert_{L^\infty}^2,
\end{aligned}
\end{equation}
which shows that
\begin{equation}\nonumber
\begin{aligned}
I_7 \lesssim &\frac{\nu^2}{|k|^2(y_2-y_1)^3\delta}\left(\frac{1}{\delta^2}\Vert w' \Vert_{L^2}^2+\frac{1}{\delta^3}\Vert w \Vert_{L^\infty}^2\right)\\
\lesssim &\frac{\nu^2}{|k|^2(y_2-y_1)^3\delta^3}\Vert w' \Vert_{L^2}^2+\frac{\nu^2}{|k|^2(y_2-y_1)^3\delta^4}\Vert w \Vert_{L^\infty}^2\\
\lesssim&\mathcal{E}_2(w).
\end{aligned}
\end{equation}

Summing up the estimates for $I_i(i=1,\cdots,7)$, we conclude that
\ben\label{eq:in-phi1-b2}
\int_{y_1+\delta}^{y_2-\delta}\frac{2|\varphi_1|^2}{(1-y^2-\lambda)^2}\mathrm{d}y \lesssim \mathcal{E}_2(w).
\een

Now \eqref{eq:in-phi1} follows from \eqref{eq:in-phi1-est1}, \eqref{eq:in-phi1-b1}, \eqref{eq:in-phi1-b0}, \eqref{eq:in-phi1-b1} and \eqref{eq:in-phi1-b2}.
\end{proof}

The following two lemmas give $L^\infty$ estimate of $\varphi_2$
and $w'$ in terms of $\|F\|_{L^2}, \|w\|_{L^2}$ and $\|w\|_{L^\infty}$.

\begin{lemma}\label{lem:varphi-Linfty}
For any $\nu \in (0,1],\lambda\in [0,1]$ and $\delta \in (0,1]$, there holds that
\begin{equation}\nonumber
\begin{aligned}
\frac{\Vert \varphi_2 \Vert_{L^\infty}^2}{(y_2-y_1+\delta)^2\delta} \lesssim \mathcal{F}_1(w),
\end{aligned}
\end{equation}
where
\begin{equation}\nonumber
\begin{aligned}
\mathcal{F}_1(w)=\frac{\Vert F\Vert_{L^2}^2}{|k|^2\delta^2(y_2-y_1+\delta)^2}+\frac{\nu\Vert F\Vert_{L^2}\Vert w \Vert_{L^2}}{|k|^2\delta^4(y_2-y_1+\delta)^2}+\delta \Vert w \Vert_{L^\infty}^2.
\end{aligned}
\end{equation}
\end{lemma}

\begin{proof}
We will follow the proof of Lemma 3.10 in \cite{LWZ2}.
It suffices to estimate $\|\varphi\|_{L^\infty\big(B(y_1,\delta)\cup B(y_2,\delta)\big)}$.
Recall (3.24) in \cite{LWZ2}: for $a,b>0$,
\beno
|\varphi(a)|\le \frac 1 {2b}\Big|\int_{a-b}^{a+b}\varphi(y)\mathrm{d}y\Big|+b^2\|\varphi''\|_{L^\infty}
\eeno
Then for any $a\in B(y_1,\delta)$, there exists $b\in [\delta/2,\delta]$ so that
\begin{align*}
|\varphi(a)|\lesssim& \frac 1 {b|k|}\Big(\int_{a-b}^{a+b}|F+\nu |k|^2w|\mathrm{d}y+\nu\big(|w'(a+b)|+|w'(a-b)|\big)\\&\qquad+\int_{a-b}^{a+b}|1-y^2-\lambda|\mathrm{d}y\|w\|_{L^\infty}\Big)+b^2\|\varphi''\|_{L^\infty}\\
\lesssim& \frac 1 {\delta|k|}\Big(\delta^\frac 12\|F\|_{L^2}+\frac \nu{\delta^\frac 12}\|w'\|_{L^2}\\&\qquad+\delta^2(y_2-y_1+\delta)\|w\|_{L^\infty}\Big)+\delta^2\|w\|_{L^\infty},
\end{align*}
which implies our result. Here we used $\|\varphi''\|_{L^\infty}\le \|w\|_{L^\infty}$ and
\beno
\int_{a-b}^{a+b}\nu |k|^2|w|\mathrm{d}y\lesssim \delta^\frac 12\nu|k|^2\|w\|_{L^2}
\lesssim \delta^\frac 12\|F\|_{L^2}.
\eeno
\end{proof}

\begin{lemma}\label{lem:dw-Linfty}
For any $\nu \in (0,1],\lambda\in [0,1]$ and $\delta \in (0,1]$, there holds that
\begin{equation}\nonumber
\begin{aligned}
\delta^3\Vert w'\Vert_{L^\infty(B(y_1,\delta)\cup B(y_2,\delta))}^2\lesssim \mathcal{F}_2(w),
\end{aligned}
\end{equation}
where
\begin{equation}\nonumber
\begin{aligned}
\mathcal{F}_2(w)=\frac{\delta^6(y_2-y_1+\delta)^2|k|^2}{\nu^2}\mathcal{F}_1(w)+\frac{\delta^4}{\nu^2}\Vert F  \Vert_{L^2}^2+\delta\|w\|_{L^\infty}^2.
\end{aligned}
\end{equation}
\end{lemma}

\begin{proof}
We follow the proof of Lemma 3.11 in \cite{LWZ2}. We have
$$\Vert w'\Vert_{L^\infty(B(y_1,\delta))}\lesssim\frac{1}{\delta}\Vert w \Vert_{L^\infty}+\Vert w''\Vert_{L^1(B(y_1,\delta))},$$
where we get by \eqref{eq:res} that
\begin{equation}\nonumber
\begin{aligned}
\delta^{\frac{3}{2}}\int_{B(y_1,\delta)}|w''|\mathrm{d}y \lesssim &\frac{\delta^{\frac{3}{2}}|k|}{\nu}\int_{B(y_1,\delta)}|(1-y^2-\lambda)w|\mathrm{d}y+\frac{\delta^{\frac{3}{2}}|k|}{\nu}\int_{B(y_1,\delta)}|\varphi|\mathrm{d}y\\
&+\frac{\delta^{\frac{3}{2}}}{\nu}\int_{B(y_1,\delta)}|F|\mathrm{d}y+\delta^{\frac{3}{2}}|k|^2\int_{B(y_1,\delta)} |w|\mathrm{d}y\\
\lesssim& \frac{\delta^3(y_2-y_1+\delta)|k|}{\nu}\delta^{\frac{1}{2}}\Vert w \Vert_{L^\infty}\\
&+\frac{(y_2-y_1+\delta)\delta^{3}|k|}{\nu}\frac{\Vert \varphi\Vert_{L^\infty(B(y_1,\delta)\cup B(y_2,\delta))}}{(y_2-y_1+\delta)\delta^{\frac{1}{2}}}+\frac{\delta^2}{\nu}\Vert F   \Vert_{L^2},
\end{aligned}
\end{equation}
here we used
\beno
\delta^{\frac{3}{2}}|k|^2\int_{B(y_1,\delta)} |w|\mathrm{d}y\lesssim \delta^{2}|k|^2\|w\|_{L^2}
\lesssim \frac{\delta^2}{\nu}\Vert F   \Vert_{L^2}.
\eeno
This shows that
$$\delta^3 \Vert w'\Vert_{L^\infty(B(y_1,\delta))}^2 \lesssim\mathcal{F}_2(w).$$
The estimate in domain $B(y_2,\delta)$ is similar.
\end{proof}

Now we are in a position to prove Proposition \ref{prop:res} in the case of $\lambda\in (0,1)$.

\begin{proof}
 We consider two cases.\smallskip

\noindent\textbf{Case 1}. $\nu^{\frac{1}{4}}|k|^{-\frac{1}{4}} \geq \frac{y_2-y_1}{4}$.\smallskip

Let $\delta=\nu^{\frac{1}{4}}|k|^{-\frac{1}{4}}<<1$. In this case, $\Vert w \Vert_{L^2(y_1,y_2)}^2 \lesssim \delta \Vert w \Vert_{L^\infty}^2 \lesssim \mathcal{E}_1(w),$ therefore $\Vert w \Vert_{L^2}^2 \lesssim \mathcal{E}_1(w).$ Thus, we only need to estimate each term in $\mathcal{E}_1(w).$

Thanks to Lemma \ref{lem:varphi-Linfty}, we have
\begin{equation}\nonumber
\begin{aligned}
\frac{\Vert \varphi_2\Vert_{L^\infty}^2}{(y_2-y_1+\delta)^2\delta} \lesssim &\frac{\Vert F\Vert_{L^2}^2}{|k|^2\delta^2 (y_2-y_1+\delta)^2}+\frac{\nu \Vert F\Vert_{L^2}\Vert w \Vert_{L^2}}{|k|^2\delta^4(y_2-y_1+\delta)^2}+\delta \Vert w \Vert_{L^\infty}^2\\
\lesssim &\frac{\Vert F \Vert_{L^2}^2}{|k|^2\delta^4}+\frac{\nu \Vert F\Vert_{L^2}\Vert w \Vert_{L^2}}{|k|^2\delta^6}+\delta \Vert w \Vert_{L^\infty}^2,
\end{aligned}
\end{equation}
which together with Lemma \ref{lem:dw-Linfty} yields that
\begin{equation}\nonumber
\begin{aligned}
\delta^3\Vert w'\Vert_{L^\infty(B(y_1,\delta)\cup B(y_2,\delta))}^2\lesssim&\frac{\delta^8|k|^2}{\nu^2}\left(\frac{\Vert F \Vert_{L^2}^2}{|k|^2\delta^4}+\frac{\nu \Vert F\Vert_{L^2}\Vert w \Vert_{L^2}}{|k|^2\delta^6}+\delta \Vert w \Vert_{L^\infty}^2\right)\\
&+\frac{\delta^4}{\nu^2}\Vert F \Vert_{L^2}^2+\delta \Vert w \Vert_{L^\infty}^2\\
\lesssim & \frac{\Vert F\Vert_{L^2}^2}{\nu |k|}+\frac{\Vert F\Vert_{L^2}\Vert w \Vert_{L^2}}{\nu^{\frac{1}{2}}|k|^{\frac{1}{2}}}+\delta\Vert w \Vert_{L^\infty}^2,
\end{aligned}
\end{equation}
which gives
\begin{equation}\nonumber
\begin{aligned}
&\frac{\nu}{|k|(y_2-y_1+\delta)\delta}\Vert w'\overline{w}\Vert_{L^\infty(B(y_1,\delta)\cup B(y_2,\delta))}\\
&\lesssim\frac{\nu}{|k|\delta^4}\delta^{\frac{3}{2}}\Vert w'\Vert_{L^\infty(B(y_1,\delta)\cup B(y_2,\delta))}\delta^{\frac{1}{2}}\Vert w \Vert_{L^\infty}\\
&\lesssim\frac{\nu}{|k|\delta^4}\left(\delta^3\Vert w'\Vert_{L^\infty(B(y_1,\delta)\cup B(y_2,\delta))}^2+\delta\Vert w \Vert_{L^\infty}^2\right)\\
&\lesssim \frac{\Vert F\Vert_{L^2}^2}{\nu |k|}+\frac{\Vert F\Vert_{L^2}\Vert w \Vert_{L^2}}{\nu^{\frac{1}{2}}|k|^{\frac{1}{2}}}+\delta \Vert w \Vert_{L^\infty}^2.
\end{aligned}
\end{equation}

Using the fact that $\nu\|w'\|_{L^2}^2\le \|F\|_{L^2}\|w\|_{L^2}$, we get
\begin{equation}\nonumber
\begin{aligned}
\frac{\nu}{|k|\delta^{\frac{5}{2}}}\Vert w'\Vert_{L^2}\Vert w \Vert_{L^\infty} &\leq \frac{\nu^2}{|k|^2 \delta^6}\Vert w'\Vert_{L^2}+\delta \Vert w \Vert_{L^\infty}^2\\
&\lesssim\frac{\nu\Vert F\Vert_{L^2}\Vert w\Vert_{L^2}}{|k|^2\delta^6}+\delta  \Vert w \Vert_{L^\infty}^2,\\
\end{aligned}
\end{equation}
and
\begin{equation}\nonumber
\begin{aligned}
&\delta  \Vert w \Vert_{L^\infty}^2 \lesssim \delta \Vert w \Vert_{L^2} \Vert w' \Vert_{L^2}\lesssim \frac{\delta}{\nu^{\frac{1}{2}}}\Vert F \Vert_{L^2}^{\frac{1}{2}}\Vert w \Vert_{L^2}^{\frac{3}{2}}.
\end{aligned}
\end{equation}

Summing up and due to $\delta=\nu^{\frac{1}{4}}|k|^{-\frac{1}{4}}$, we conclude that
\begin{equation}\nonumber
\begin{aligned}
\Vert w \Vert_{L^2}^2\lesssim \nu^{-1}|k|^{-1}\Vert F \Vert_{L^2}^2.
\end{aligned}
\end{equation}

\noindent\textbf{Case 2. $\nu^{\frac{1}{4}}|k|^{-\frac{1}{4}} \leq \frac{y_2-y_1}{4}.$}
Take $\delta^3(y_2-y_1)|k|=\nu$, then $0<\delta \leq \frac{y_2-y_1}{4}$.
In this case, we have $\mathcal{E}_1(w)\le \mathcal{E}_2(w)$. We deduce from Lemma \ref{lem:outside} and Lemma \ref{lem:inside} that $\Vert w \Vert_{L^2}^2 \lesssim \mathcal{E}_2(w).$ Now we estimate each term in $\mathcal{E}_2(w)$.
First of all, we have
\beno
&&\frac{\nu^2}{|k|^2(y_2-y_1)^3\delta^4}\Vert w \Vert_{L^\infty}^2\leq \frac{\delta^6 (y_2-y_1)^2}{(y_2-y_1)^3\delta^4}\Vert w \Vert_{L^\infty}^2 \lesssim \delta\Vert w \Vert_{L^\infty}^2,\\
&&\frac{\Vert F \Vert_{L^2}\Vert w \Vert_{L^2}}{|k|\delta(y_2-y_1)}\lesssim \frac{\Vert F \Vert_{L^2}\Vert w \Vert_{L^2}}{|k|\delta^\frac 32(y_2-y_1)^\frac 12}\lesssim \frac{\Vert F \Vert_{L^2}\Vert w \Vert_{L^2}}{\nu^{\frac{1}{2}}|k|^{\frac{1}{2}}},\\
&&\frac{\Vert F \Vert_{L^2}^2}{|k|^2(y_2-y_1)^3\delta}\lesssim \frac{\Vert F \Vert_{L^2}^2}{\nu|k|},\\
&&\frac{\nu^2}{|k|^2}\frac{\Vert w'\Vert_{L^2}^2}{(y_2-y_1)^3\delta^3}\lesssim \frac{\Vert F\Vert_{L^2}\Vert w \Vert_{L^2}}{|k|(y_2-y_1)^2}\lesssim \frac{\Vert F\Vert_{L^2}\Vert w \Vert_{L^2}}{\nu^{\frac{1}{2}}|k|^{\frac{1}{2}}}.
\eeno
Thanks to $\delta\le |k|^{-\frac 14}|\nu|^{\frac 14}$, we get
$$\delta  \Vert w \Vert_{L^\infty}^2 \lesssim \delta \Vert w \Vert_{L^2} \Vert w' \Vert_{L^2}\lesssim \frac{\delta}{\nu^{\frac{1}{2}}}\Vert F \Vert_{L^2}^{\frac{1}{2}}\Vert w \Vert_{L^2}^{\frac{3}{2}}\lesssim \frac{1}{\nu^{\frac{1}{4}}|k|^{\frac{1}{4}}}\Vert F \Vert_{L^2}^{\frac{1}{2}}\Vert w \Vert_{L^2}^{\frac{3}{2}},$$
which gives
\begin{equation}\nonumber
\begin{aligned}
\frac{\nu}{|k|\delta^{\frac{3}{2}}(y_2-y_1)}\Vert w'\Vert_{L^2}\Vert w \Vert_{L^\infty}=&\delta\Vert w'\Vert_{L^2}\delta^{\frac{1}{2}}\Vert w \Vert_{L^\infty} \lesssim\delta^2 \Vert w'\Vert_{L^2}^2+\delta \Vert w \Vert_{L^\infty}\\
\lesssim&\frac{\delta^2}{\nu}\Vert F\Vert_{L^2}\Vert w \Vert_{L^2}+\delta \Vert w \Vert_{L^\infty}^2\\
\lesssim &\frac{\Vert F \Vert_{L^2}\Vert w \Vert_{L^2}}{\nu^{\frac{1}{2}}|k|^{\frac{1}{2}}}+\frac{1}{\nu^{\frac{1}{4}}|k|^{\frac{1}{4}}}\Vert F \Vert_{L^2}^{\frac{1}{2}}\Vert w \Vert_{L^2}^{\frac{3}{2}}.
\end{aligned}
\end{equation}

Thanks to Lemma \ref{lem:varphi-Linfty}, we have
\begin{equation}\nonumber
\begin{aligned}
\frac{\Vert \varphi_2\Vert_{L^\infty}^2}{(y_2-y_1)^2\delta}\lesssim&\frac{\Vert F\Vert_{L^2}^2}{|k|^2\delta^2(y_2-y_1)^2}+\frac{\nu\Vert F\Vert_{L^2}\Vert w \Vert_{L^2}}{|k|^2\delta^4(y_2-y_1)^2}+\delta \Vert w \Vert_{L^\infty}^2\\
\lesssim &\frac{\Vert F\Vert_{L^2}^2}{\nu|k|}+\frac{\Vert F\Vert_{L^2}\Vert w \Vert_{L^2}}{\nu^{\frac{1}{2}}|k|^{\frac{1}{2}}}+\frac{1}{\nu^{\frac{1}{4}}|k|^{\frac{1}{4}}}\Vert F \Vert_{L^2}^{\frac{1}{2}}\Vert w \Vert_{L^2}^{\frac{3}{2}}.
\end{aligned}
\end{equation}
Here we used $\frac {\nu} {|k|^2\delta^4(y_2-y_2)^2}\le \frac {1} {|k|\delta (y_2-y_2)}\le \frac {1} {|k|^\frac 12 \nu^\frac 12}$. Then by Lemma \ref{lem:dw-Linfty}, we get
\begin{equation}\nonumber
\begin{aligned}
\delta^3\Vert w'\Vert_{L^\infty(B(y_1,\delta)\cup B(y_2,\delta))}^2\lesssim &\frac{\delta^6(y_2-y_1)^2|k|^2}{\nu^2}\mathcal{F}_1(w)+\frac{\delta^4}{\nu^2}\Vert F  \Vert_{L^2}^2+\delta\|w\|^2_{L^\infty}\\
\lesssim &\frac{\Vert F\Vert_{L^2}^2}{\nu|k|}+\frac{\Vert F\Vert_{L^2}\Vert w \Vert_{L^2}}{\nu^{\frac{1}{2}}|k|^{\frac{1}{2}}}+\frac{1}{\nu^{\frac{1}{4}}|k|^{\frac{1}{4}}}\Vert F \Vert_{L^2}^{\frac{1}{2}}\Vert w \Vert_{L^2}^{\frac{3}{2}},
\end{aligned}
\end{equation}
which gives
\begin{equation}\nonumber
\begin{aligned}
&\frac{\nu}{|k|\delta (y_2-y_1)}\Vert w'\overline{w}\Vert_{L^\infty(B(y_1,\delta)\cup B(y_2,\delta))}\\
&\lesssim\delta^{\frac{3}{2}}\Vert w'\Vert_{L^\infty(B(y_1,\delta)\cup B(y_2,\delta))}\delta^{\frac{1}{2}}\Vert w\Vert_{L^\infty}\\
&\lesssim\delta^3\Vert w'\Vert_{L^\infty(B(y_1,\delta)\cup B(y_2,\delta))}^2+\delta \Vert w \Vert_{L^\infty}^2\\
&\lesssim\frac{\Vert F\Vert_{L^2}^2}{\nu|k|}+\frac{\Vert F\Vert_{L^2}\Vert w \Vert_{L^2}}{\nu^{\frac{1}{2}}|k|^{\frac{1}{2}}}+\frac{1}{\nu^{\frac{1}{4}}|k|^{\frac{1}{4}}}\Vert F \Vert_{L^2}^{\frac{1}{2}}\Vert w \Vert_{L^2}^{\frac{3}{2}},
\end{aligned}
\end{equation}
and
\begin{equation}\nonumber
\begin{aligned}
\frac{\nu^2}{|k|^2}\frac{\Vert w'\Vert_{L^\infty(B(y_1,\delta)\cup B(y_2,\delta))}^2}{(y_2-y_1)^3\delta^2}\lesssim&\delta^3\Vert w'\Vert_{L^\infty(B(y_1,\delta)\cup B(y_2,\delta))}^2\\
\lesssim &\frac{\Vert F\Vert_{L^2}^2}{\nu|k|}+\frac{\Vert F\Vert_{L^2}\Vert w \Vert_{L^2}}{\nu^{\frac{1}{2}}|k|^{\frac{1}{2}}}+\frac{1}{\nu^{\frac{1}{4}}|k|^{\frac{1}{4}}}\Vert F \Vert_{L^2}^{\frac{1}{2}}\Vert w \Vert_{L^2}^{\frac{3}{2}}.
\end{aligned}
\end{equation}

Summing up, we conclude that
\begin{equation}\nonumber
\begin{aligned}
\Vert w \Vert_{L^2}^2\lesssim \nu^{-1}|k|^{-1}\Vert F \Vert_{L^2}^2.
\end{aligned}
\end{equation}

This completes the proof of Proposition \ref{prop:res} in the case of $
\lambda\in [0,1]$.
\end{proof}

\subsection{Case of $\lambda\le 0$}
First of all, we get by integration by parts that
\begin{equation}\nonumber
\begin{aligned}
&\Big|\mathrm{Im} \big\langle -\nu(\partial_y^2-|k|^2)w+ik[(1-y^2-\lambda)w+2\varphi], w\big\rangle\Big|\\
&\geq |k|\left(\int_{-1}^{1}(1-y^2-\lambda)|w|^2\mathrm{d}y+2 \int_{-1}^{1}\varphi\overline{w}\mathrm{d}y\right),
\end{aligned}
\end{equation}
which gives
\begin{equation}\nonumber
\begin{aligned}
\int_{-1}^1(1-y^2-\lambda)|w|^2\mathrm{d}y+\big\langle 2\varphi,w\big\rangle \leq |k|^{-1}\Vert F\Vert_{L^2}\Vert w \Vert_{L^2}.
\end{aligned}
\end{equation}
This along with \eqref{eq:energy-key} shows that
\begin{equation}\label{eq:lam-n1}
\begin{aligned}
2\int_{-1}^1(1-y^2-\lambda)^2&\bigg|\left(\frac{\varphi}{1-y^2-\lambda}\right)'\bigg|^2\mathrm{d}y\\
&+\int_{-1}^1|k|^2|\varphi|^2\mathrm{d}y
\leq |k|^{-1}\Vert F\Vert_{L^2}\Vert w \Vert_{L^2}.
\end{aligned}
\end{equation}

On the other hand, we have
\begin{equation}\nonumber
\begin{aligned}
&\bigg|\mathrm{Im} \bigg\langle -\nu(\partial_y^2-|k|^2)w\\
&\qquad+ik[(1-y^2-\lambda)w+2\varphi], \frac{w}{1-y^2-\lambda} \chi_{(-1+\delta,1-\delta)}\bigg\rangle\bigg|\\
&\geq -2\nu\Vert w'\overline{w}\Vert_{L^\infty\big((-1,-1+\delta)\cup(1-\delta,1)\big)}\frac{1}{|1-(1-\delta)^2-\lambda|}\\
&\qquad-\nu\Vert w \Vert_{L^\infty}\Vert w' \Vert_{L^2}\left\Vert \frac{2y}{(1-y^2-\lambda)^2}\right\Vert_{L^2(-1+\delta,1-\delta)}\\
&\qquad+\frac{|k|}{2}\int_{-1+\delta}^{1-\delta}|w|^2\mathrm{d}y-|k|\int_{-1+\delta}^{1-\delta}\frac{2|\varphi|^2}{(1-y^2-\lambda)^2}\mathrm{d}y,
\end{aligned}
\end{equation}
which gives
\begin{align*}
&\frac{|k|}{2}\int_{-1+\delta}^{1-\delta}|w|^2\mathrm{d}y \\
&\leq \Vert F\Vert_{L^2}\left\Vert \frac{w}{(1-y^2-\lambda)}\right\Vert_{L^2(-1+\delta,1-\delta)}+\frac{2\nu \Vert w'\overline{w}\Vert_{L^\infty((-1,-1+\delta)\cup(1-\delta,1))}}{|1-(1-\delta)^2-\lambda|}\\
&\qquad+\nu\Vert w \Vert_{L^\infty}\Vert w' \Vert_{L^2}\left\Vert \frac{2y}{(1-y^2-\lambda)^2}\right\Vert_{L^2(-1+\delta,1-\delta)}\\
&\qquad+|k|\int_{-1+\delta}^{1-\delta}\frac{2|\varphi|^2}{(1-y^2-\lambda)^2}\mathrm{d}y.
\end{align*}
Thus, we obtain
\begin{equation}\nonumber
\begin{aligned}
\Vert w\Vert_{L^2}^2 \lesssim &\Vert w\Vert_{L^2(-1+\delta,1-\delta)}^2+\delta\Vert w \Vert_{L^\infty}^2\\
\lesssim& \frac{\Vert F\Vert_{L^2}\Vert w \Vert_{L^2}}{|k|\delta}+\frac{\nu \Vert w\Vert_{L^\infty}\Vert w'\Vert_{L^\infty\left(A\right)}}{|k|(1-(1-\delta)^2-\lambda)}\\
&+\frac{\nu}{\delta^{\frac{3}{2}}|k|}\Vert w'\Vert_{L^2}\Vert w \Vert_{L^\infty}+\delta \Vert w \Vert_{L^\infty}^2+\int_{-1+\delta}^{1-\delta}\frac{2|\varphi|^2}{(1-y^2-\lambda)^2}\mathrm{d}y,
\end{aligned}
\end{equation}
where $A=(-1,-1+\delta)\cup (1-\delta,1)$.

We get by \eqref{eq:lam-n1} that
\begin{equation}\label{4}
\begin{aligned}
\int_{-1+\delta}^{1-\delta}\frac{|\varphi|^2}{(1-y^2-\lambda)^2}\mathrm{d}y\lesssim \delta^{-2}\Vert \varphi\Vert_{L^2}^2\lesssim \frac{1}{|k|\delta^2}\Vert F\Vert_{L^2}\Vert w \Vert_{L^2}.
\end{aligned}
\end{equation}
From the proof of Lemma \ref{lem:dw-Linfty} and Lemma \ref{lem:varphi-Linfty}, we infer that
\begin{equation}\nonumber
\begin{aligned}
&\frac{\Vert w'\Vert_{L^\infty\left(A\right)}}{|1-(1-\delta)^2-\lambda|}\\
&\lesssim \frac{|k|}{\nu}\delta\Vert w \Vert_{L^\infty}+\frac{|k|}{|1-(1-\delta)^2-\lambda|\nu}\delta\Vert \varphi\Vert_{L^\infty(A)}\\
&\qquad+\frac{1}{\delta^{\frac{1}{2}}\nu}\Vert F   \Vert_{L^2}+\delta^{-2}\|w\|_{L^\infty}\\
&\lesssim \frac{|k|}{\nu}\delta\Vert w \Vert_{L^\infty}+\nu^{-1}\big(\delta^{-\frac 12}\|F\|_{L^2}+\nu\delta^{-\frac 32}\|w'\|_{L^2}+\delta\|w\|_{L^\infty}+|k|\delta^2\|w\|_{L^\infty}\big)\\
&\qquad+\frac{1}{\delta^{\frac{1}{2}}\nu}\Vert F   \Vert_{L^2}+\delta^{-2}\|w\|_{L^\infty}
\end{aligned}
\end{equation}

Summing up, we conclude that
\begin{equation}\nonumber
\begin{aligned}
\Vert w\Vert_{L^2}^2
\lesssim& \frac{\Vert F\Vert_{L^2}\Vert w \Vert_{L^2}}{|k|\delta^2}+\|w\|_{L^\infty}|k|^{-1}\Big(|k|\delta\|w\|_{L^\infty}+\delta^{-\frac 12}\|F\|_{L^2}\\
&+\nu\delta^{-\frac 32}\|w'\|_{L^2}+\delta\|w\|_{L^\infty}+|k|\delta^2\|w\|_{L^\infty}+\nu\delta^{-2}\|w\|_{L^\infty}\Big)\\
&+\frac{\nu}{\delta^{\frac{3}{2}}|k|}\Vert w'\Vert_{L^2}\Vert w \Vert_{L^\infty}+\delta \Vert w \Vert_{L^\infty}^2\\
\lesssim& \frac{\Vert F\Vert_{L^2}\Vert w \Vert_{L^2}}{|k|\delta^2}+\delta \Vert w \Vert_{L^\infty}^2
+|k|^{-1}\delta^{-\frac 12}\|F\|_{L^2}\|w\|_{L^\infty}\\
&+\nu\delta^{-\frac 32}|k|^{-1}\|w'\|_{L^2}\|w\|_{L^\infty}+\nu \delta^{-2}|k|^{-1}\|w\|_{L^\infty}^2.
\end{aligned}
\end{equation}
Taking $\delta=\nu^{\frac 14}|k|^{-\frac 14}$, we obtain
\begin{equation}\nonumber
\begin{aligned}
\Vert w\Vert_{L^2}^2
&\lesssim \frac{\Vert F\Vert_{L^2}\Vert w \Vert_{L^2}}{|k|^\frac 12\nu^\frac 12}+\delta \Vert w \Vert_{L^\infty}^2
+\delta^{-2}|k|^{-2}\|F\|_{L^2}^2+\nu^2\delta^{-4}|k|^{-2}\|w'\|_{L^2}^2\\
&\lesssim \frac{\Vert F\Vert_{L^2}\Vert w \Vert_{L^2}}{|k|^\frac 12\nu^\frac 12}+\nu^{-\frac 14}|k|^{-\frac 14}\Vert w \Vert_{L^2}^\frac 32\|F\|_{L^2}^\frac 12
+\nu^{-\frac12}|k|^{-\frac 32}\|F\|_{L^2}^2,
\end{aligned}
\end{equation}
which implies that
\beno
\|w\|_{L^2}\lesssim (\nu|k|)^{-\frac 12}\|F\|_{L^2}.
\eeno

The proof is completed.

\section{Proof of the main Theorems: enhanced dissipation and nonlinear stability}\label{enhanceddecay}
In this section, we will prove our main results.
\subsection{The Enhanced dissipation}
An operator $H$ in a Hilbert space $X$ is accretive if $\mathrm{Re}\langle Hf,f\rangle \geq0$ for any $f\in X$ and an accretive operator is called m-accretive if any $\lambda <0$ belongs to the resolvent set of $H$.  We define
$$\Psi(A)=\inf\big\{\Vert (A-i\lambda)u\Vert:u\in D(A),\lambda \in \mathbb{R},\Vert u\Vert=1\big\}.$$

The following Gearhart-Pr$\ddot{\mathrm{u}}$ss type lemma comes from \cite{Wei}.
\begin{lemma}\label{lem:GP}
Let $A$ be an m-accretive operator on a Hilbert spaces $X$. Then  for any $t>0$,
\beno
\Vert e^{-tA}\Vert \leq e^{-t\Psi(A)+\frac{\pi}{2}}.
\eeno
\end{lemma}

Recall that
\beno
\widehat{\mathscr{L}}=-\nu(\partial_y^2-|k|^2)+ik\big[(1-y^2)+2(\partial_y^2-|k|^2)^{-1}\big],
\eeno
It is easy to see that $\widehat{\mathscr{L}}$ is an m-accretive operator. Moreover, it follows from Proposition \ref{prop:res} that
\begin{equation}\nonumber
\begin{aligned}
\Psi(\widehat{\mathscr{L}}) \geq c\nu^{\frac{1}{2}}|k|^{\frac{1}{2}}+\nu |k|^2.
\end{aligned}
\end{equation}
Notice that
\begin{equation}\nonumber
\begin{aligned}
\Vert e^{-t\mathscr{L}} g_{\not=}\Vert_{L^2}^2&\sim\sum_{k\not=0}\Vert e^{-\widehat{\mathscr{L}}t}\widehat{g}(t;k,y)\Vert_{L^2}^2.
\end{aligned}
\end{equation}
Then Lemma \ref{lem:GP} gives the following enhanced dissipation estimate.

\begin{proposition}\label{prop:ED}
There exist constants $C,c>0$, independent of $\nu$ such that
\begin{equation}\nonumber
\begin{aligned}
\Vert e^{-t\mathscr{L}} g_{\not=}\Vert_{L^2} \leq Ce^{-c\nu^{\frac{1}{2}}t-\nu t}\Vert g_{\not=}\Vert_{L^2}.
\end{aligned}
\end{equation}
\end{proposition}
The proof of Proposition \ref{prop:ED} is similar to that in \cite{LWZ1} and here we omit it.

And the Theorem \ref{thm:ED} follows from Proposition \ref{prop:ED}.

\subsection{Nonlinear stability threshold}\label{nonlinearstability}
Now we consider the nonlinear problem
\begin{equation}\label{6.1}
\left \{
\begin{array}{lll}
(\partial_t -\nu\Delta+(1-y^2)\partial_x)\omega+2\partial_x \Delta^{-1}\omega=\nabla\cdot f,\\
\omega(t;x,\pm 1)=0,\\
\omega_{t=0}=\omega_0=:\omega(0),
\end{array}
\right.
\end{equation}
where $f=-u\omega$.

To derive the space-time estimates, we decompose the vorticity  $\omega=\omega^L+\omega^{NL}$ with
\begin{equation}\label{6.2}
\left \{
\begin{array}{lll}
(\partial_t+\mathscr{L})\omega^{NL}=\nabla\cdot f,\\
\omega^{NL}|_{t=0}=0,\\
\omega^{NL}(t;x,\pm 1)=0,
\end{array}
\right.
\end{equation}
and
\begin{equation}\label{6.3}
\left \{
\begin{array}{lll}
(\partial_t+\mathscr{L})\omega^{L}=0,\\
\omega^{L}|_{t=0}=\omega_0\triangleq \omega(0),\\
\omega^{L}(t;x,\pm 1)=0,
\end{array}
\right.
\end{equation}
where $\mathscr{L}=-\nu \Delta+(1-y^2)\partial_x+2\partial_x \Delta^{-1}$.
\begin{lemma}
There exists a constant $C>0$ independent of $t$ such that
\begin{equation}\label{6.4}
\Vert (\omega^{NL})_{\not=}(t)\Vert_{L^2}^2+\nu \int_0^t\Vert (\nabla \omega^{NL})_{\not=}(s)\Vert_{L^2}^2 \mathrm{d}s \leq C\nu^{-1}\int_0^t \Vert f_{\not=}(s)\Vert_{L^2}^2\mathrm{d}s,
\end{equation}
\begin{equation}\label{6.5}
\left\Vert \int_0^t e^{-(t-s)\mathscr{L}}(\nabla \cdot \ f_{\not=}(s))\mathrm{d}s\right\Vert_{L^2}^2\leq C\nu^{-1}\int_0^t \Vert f_{\not=}(s)\Vert_{L^2}^2\mathrm{d}s.
\end{equation}
\end{lemma}
\begin{proof}
The proof of this lemma is similar to Lemma 6.1 of \cite{LWZ1}, here we omit it.
\end{proof}
\begin{lemma}
Let $c$ be in Proposition \ref{prop:ED} and $c'\in(0,c)$, then it holds that
\begin{equation}\label{6.7}
\begin{array}{lll}
&\Vert e^{c'\nu^{\frac{1}{2}}t}\omega^{NL}_{\not=}\Vert_{L^\infty L^2}^2+\nu^{\frac{1}{2}}\Vert e^{c'\nu^{\frac{1}{2}}t}\omega^{NL}_{\not=}\Vert_{L^2L^2}^2\\
&\qquad+\nu \Vert e^{c'\nu^{\frac{1}{2}}t}(\nabla\omega^{NL})_{\not=}\Vert_{L^2L^2}^2
\lesssim \nu^{-1} \Vert e^{c'\nu^{\frac{1}{2}}t}f_{\not=}\Vert_{L^2L^2}^2.
\end{array}
\end{equation}
\end{lemma}
\begin{proof}
This lemma can be obtained by applying the similar arguments as in the Lemma 6.2 of \cite{LWZ1}, here we omit the proof.
\end{proof}
Therefore, we conclude that
\begin{proposition}\label{pro:sp-time}
Let $c$ be in Proposition \ref{prop:ED} and $c'\in(0,c)$, then it holds that
\begin{equation}\label{6.8}
\begin{array}{lll}
&\Vert e^{c'\nu^{\frac{1}{2}}t}\omega_{\not=}\Vert_{L^\infty L^2}^2+\nu^{\frac{1}{2}}\Vert e^{c'\nu^{\frac{1}{2}}t}\omega_{\not=}\Vert_{L^2L^2}^2\\
&+\nu \Vert e^{c'\nu^{\frac{1}{2}}t}(\nabla\omega)_{\not=}\Vert_{L^2L^2}^2
\lesssim \|\omega_{\not=}(0)\|_{L^2}^2+\nu^{-1} \Vert e^{c'\nu^{\frac{1}{2}}t}f_{\not=}\Vert_{L^2L^2}^2.
\end{array}
\end{equation}
\end{proposition}
\begin{proof}
The proof of this proposition is similar to Proposition 6.3 of \cite{LWZ1} and here we omit it.
\end{proof}

Now we are on a position to prove Theorem \ref{mainresult1}.
\begin{proof}[Proof of Theorem \ref{mainresult1}]
First, we have the following basic estimate
$$\|\omega\|_{L^\infty L^2}^2+\nu\|\nabla \omega\|_{L^2L^2}^2\leq \|\omega_0\|_{L^2}^2.$$
It follows from Proposition \ref{pro:sp-time} that
\begin{equation}\nonumber
\begin{array}{lll}
&\Vert e^{c'\nu^{\frac{1}{2}}t}\omega_{\not=}\Vert_{L^\infty L^2}^2+\nu^{\frac{1}{2}}\Vert e^{c'\nu^{\frac{1}{2}}t}\omega_{\not=}\Vert_{L^2L^2}^2\\
&+\nu \Vert e^{c'\nu^{\frac{1}{2}}t}(\nabla\omega)_{\not=}\Vert_{L^2L^2}^2
\lesssim \|\omega_{\not=}(0)\|_{L^2}^2+\nu^{-1} \Vert e^{c'\nu^{\frac{1}{2}}t}f_{\not=}\Vert_{L^2L^2}^2.
\end{array}
\end{equation}
It remains to estimate the nonlinear term $f_{\not=}$.
Recall that $(\partial_y^2-|k|^2)\varphi=\widehat{\omega}, \widehat{u}=(\partial_y\varphi,-ik\varphi)$ and
\begin{equation}\nonumber
\begin{aligned}
\|u_{\not=}\|_{L^\infty}^2\thicksim\sum_{k\not=0}\|\widehat{u}(t;k,\cdot)\|_{L^\infty}^2,
\end{aligned}
\end{equation}
and
\begin{equation}\nonumber
\begin{aligned}
\|\widehat{u}(t;k,\cdot)\|_{L^\infty}^2\leq & \|\widehat{u}(t;k,\cdot)\|_{L^2}\|\widehat{u}(t;k,\cdot)\|_{H^1}\\
\leq&(\|\partial_y\varphi\|_{L^2}^2+|k|^2\|\varphi\|_{L^2}^2)^{\frac{1}{2}}
(\|\partial_y^2\varphi\|_{L^2}^2+2|k|^2\|\partial_y\varphi\|_{L^2}^2+|k|^4\|\varphi\|_{L^2}^2)^{\frac{1}{2}}\\
\lesssim&\|\widehat{\omega}(t;k,\cdot)\|_{L^2}^2,
\end{aligned}
\end{equation}
therefore, one has
$$\|u_{\not=}\|_{L^\infty}\lesssim \|\omega_{\not=}\|_{L^2}.$$
Note that
$$(u\omega)_{\not=}=\overline{u}\omega_{\not=}+u_{\not=}\overline{\omega}+(u_{\not=}\omega_{\not=})_{\not=},$$
then we obtain that
\begin{equation}\nonumber
\begin{aligned}
 \Vert e^{c'\nu^{\frac{1}{2}}t}f_{\not=}\Vert_{L^2L^2}^2 \leq &\Vert e^{c'\nu^{\frac{1}{2}}t}\overline{u}\omega_{\not=}\Vert_{L^2L^2}^2+\Vert e^{c'\nu^{\frac{1}{2}}t}u_{\not=}\overline{\omega}\Vert_{L^2L^2}^2+
 \Vert e^{c'\nu^{\frac{1}{2}}t}(u_{\not=}\omega_{\not=})_{\not=}\Vert_{L^2L^2}^2\\
 \lesssim&\|\overline{u}\|_{L^\infty L^\infty}^2\|e^{c'\nu^{\frac{1}{2}}t}\omega_{\not=}\Vert_{L^2L^2}^2
 +\|\overline{\omega}\|_{L^\infty L^2}^2\|e^{c'\nu^{\frac{1}{2}}t}u_{\not=}\Vert_{L^2L^2}^2
 \\
 &+\|\omega_{\not=}\|_{L^\infty L^2}^2\|e^{c'\nu^{\frac{1}{2}}t}u_{\not=}\Vert_{L^2L^2}^2\\
 \lesssim& \|\omega_0\|_{L^2}^2\nu^{-\frac{1}{2}}\Vert \omega_{\not=}\Vert_{X_{c'}}^2.
\end{aligned}
\end{equation}
Moreover, we have
\begin{equation}\nonumber
\begin{aligned}
\Vert\omega_{\not=}\Vert_{X_{c'}}^2\lesssim  \|\omega_{\not=}(0)\|_{L^2}^2+\|\omega_0\|_{L^2}^2\nu^{-\frac{3}{2}}\Vert \omega_{\not=}\Vert_{X_{c'}}^2.
\end{aligned}
\end{equation}
Therefore, if $\|\omega_0\|_{L^2}\leq c_0\nu^{\frac{3}{4}}$ for some small $c_0>0$, we deduce that
\begin{equation}\nonumber
\begin{aligned}
\Vert\omega_{\not=}\Vert_{X_{c'}}^2 \leq C\|\omega_{\not=}(0)\|_{L^2}^2,
\end{aligned}
\end{equation}
then the Theorem \ref{mainresult1} follows from a continuity argument.
\end{proof}

\section*{Acknowledgment}
The authors thank to Prof. Zhifei Zhang for suggesting this problem and many valuable discussions. Ding is supported by NSFC (11371152, 11571117, 11871005 and 11771155) and NSF of Guangdong (2017A030313003). Part of this work was conducted during Lin visited the School of Mathematical Science at Peking University, Lin would like to thank the school for their hospitality.


\begin{thebibliography}{99}

\bibitem{Beck} M. Beck, C. E. Wayne, {\it  Metastability and rapid convergence to quasi-stationary bar states for the two-dimensional Navier-Stokes equations}, Proc. Roy. Soc. Edinburgh Sect. A: Mathematics., 143(2013), pp. 905-927.

\bibitem{BGM1}
J. Bedrossian, P. Germain, N. Masmoudi, {\it Dynamics near the subcritical transition of the 3D Couette flow I: Below threshold case,} Mem. of the AMS., 266(2020). https://doi.org/10.1090/memo/1294.


\bibitem{BGM2} J. Bedrossian, P. Germain, N. Masmoudi, {\it Dynamics near the subcritical transition of the 3D Couette flow II: Above threshold case,} arXiv:1506.03721, to appear in Mem. of the AMS.

\bibitem{BGM3} J. Bedrossian, P. Germain, N. Masmoudi, {\it On the stability threshold for the 3D Couette flow in Sobolev regularity}, Ann. Math., 185(2017), pp. 541-608.

\bibitem{BGM4}
    J. Bedrossian, P. Germain, N. Masmoudi, {\it Stability of the Couette flow at high Reynolds numbers in two dimensions and three dimensions}, Bulletin of the American Mathematical Society, 56(2019), pp. 373-414.


\bibitem{BMV}
    J. Bedrossian, N. Masmoudi, V. Vicol, {\it Enhanced dissipation and inviscid damping in the inviscid limit of the Navier-Stokes equations near the two dimensional Couette flow}, Arch. Ration. Mech. Anal., 219(2016), 1087-1159.

\bibitem{BVW}
    J. Bedrossian, V. Vicol, F. Wang, {\it The Sobolev stability threshold for 2D shear flows near Couette}, J. Nonlinear Sci., 28(2018), pp. 2051-2075.

\bibitem{CLWZ}
Q. Chen, T. Li, D. Wei, Z. Zhang, {\it Transition threshold for the 2-D Couette flow in a finite channel}, Arch. Rational Mech. Anal., 238(2020), pp. 125-183.

\bibitem{CWZ1} Q. Chen, D. Wei, Z. Zhang, {\it Linear stability of pipe Poiseuille flow at high Reynolds number regime}, arXiv:1910.14245, 2019.

\bibitem{CWZ2} Q. Chen, D. Wei, Z. Zhang, {\it Transition threshold for the 3D Couette flow in a finite channel}, arXiv:2006.00721, 2020.

\bibitem{Coti}
M. Coti Zelati, T. M. Elgindi, K. Widmayer, {\it Enhanced dissipation in the Navier-Stokes equations near the Poiseuille flow}, Commun. Math. Phys., 378(2020), pp. 987-1010 .


%
%

\bibitem{Chapman}
S. J. Chapman, {\it Subcritial transition in chanel flow}, J. Fluid Mech., 451(2002), pp. 35-91.

\bibitem{DHB}
F. Daviaud, J. Hagseth, P. Berg$\acute{\mathrm{e}}$, {\it Subcritical transition to turbulence in plane Couette flow}, Phys. Rev. Lett., 69(1992), pp. 2511-2514.

\bibitem{Drazin}
P. G. Drazin, W. H. Reid, {\it Hydrodynamic stability}, Cambtrdge Monographs on Mechanics and Applied Mathematics, Cambridge University Press, Cambridge-New York, 1981.

\bibitem{Gall}
T. Gallay, {\it Enhanced dissipation and axisymmetrization of two-dimensional viscous vortices}, Arch. Ration. Mech. Anal., 230(2018), pp. 939-975.

\bibitem{IMM}
S. Ibrahim, Y. Maekawa, N. Masoudi, {\it On pseudospectral bound for non-selfadjoint operators and its application to stability of Kolmogorov flows}, Ann. PDE., 5 (2019). https://doi.org/10.1007/s40818-019-0070-7.

\bibitem{Kato}
T. Kato, {\it Perturbation theory for linear operators}, Die Grundlehren der mathematischen Wissenschaften, 132, Springer New York, New York, 1966.

\bibitem{Kelvin}
L. Kelvin, {\it Stability of fluid motion-rectilinear motion of viscous fluid between two parallel plates}, Phil. Mag., 24(1887), pp. 188-196.

\bibitem{LWZ1}
T. Li, D. Wei, Z. Zhang, {\it Pseudospectral and spectral bounds for the Oseen vortices operator}, arXiv: 1701.06269, 2017.

\bibitem{LWZ2}
T. Li, D. Wei, Z. Zhang, {\it Pseudospectral Bound and Transition Threshold for the 3D Kolmogorov Flow}, Comm. Pure Appl. Math., 73(2020), pp. 465-557.

\bibitem{Li}
Y. C. Li, Z. Lin, {\it A resolution of the Sommerfeld paradox}, SIAM J. Math. Anal., 43(2011), pp. 1923-1954.


\bibitem{LX}
Z. Lin, M. Xu, {\it Metastability of Kolmogorov flows and inviscid damping of shear flows}, Arch. Ration. Mech. Anal., 231(2019), pp. 1811-1852.

\bibitem{LHR}
A. Lundbladh, D. S. Henningson, S. C. Reddy, {\it Threshold amplitudes for transition in
channel flows}, Transition, Turbulence and Combustion, 309-318. Springer, Dordrecht, 1994.



\bibitem{Orszag}
S. Orszag, L. Kells, {\it Transition to turbulence in plane Poiseuille and plane Couette flow,} J. Fluid
Mech., 96(1980), pp. 159-205.

\bibitem{Pazy}
A. Pazy, {\it Semigroups of linear operators and applications to partial differential equations}, Applied
Mathematical Sciences, 44, Springer, New York, 1983.

\bibitem{Reddy}
S. C. Reddy, P. J. Schmid, J. S. Baggett, D. S. Henningson, {\it On stability of streamwise
streaks and transition thresholds in plane channel flows}, J. Fluid Mech., 365(1998), pp. 269-303.

\bibitem{Rey}
O. Reynolds, {\it III. An experimental investigation of the circumstances which determine whether the motion of water shall be direct or sinuous, and of the law of resistance in parallel channels},
Proc. R. Soc. Lond., 35(1883), pp. 84-99.

\bibitem{Rom}
V. A. Romanov, {\it Stability of plane-parallel Couette flow}, Fun. Anl. Appl., 7(1973), pp. 137-146.

\bibitem{Sch}
P. J. Schmid, D. S. Henningson, {\it Stability and transition in shear flows}, Applied Mathematical
Sciences, 142, Springer, New York, 2001.


\bibitem{Tre}
L. N. Trefethen, A. E. Trefethen, S. C. Reddy, T. A. Driscoll, {\it Hydrodynamic stability without
eigenvalues}, Science, 261(1993), pp. 578-584.

\bibitem{Wei}
D. Wei, {\it Diffusion and mixing in fluid flow via the resolvent estimate}, Sci. China. Math., (2019). https://doi.org/10.1007/s11425-018-9461-8.

\bibitem{WZ} D. Wei, Z. Zhang, {\it Transition threshold for the 3D Couette flow in Sobolev space}, arXiv:1803.01359, 2018.


\bibitem{WZZ1}
D. Wei, Z. Zhang, W. Zhao, {\it Linear inviscid damping for a class of monotone shear flow in
Sobolev spaces,} Comm. Pure Appl. Math., 71(2018), pp. 617-687.

\bibitem{WZZ2}
D. Wei, Z. Zhang, W. Zhao, {\it Linear inviscid damping and vorticity depletion for shear flows},
 Ann. PDE., 5 (2019). https://doi.org/10.1007/s40818-019-0060-9.

 \bibitem{WZZ3}
D. Wei, Z. Zhang, W. Zhao, {\it Linear inviscid damping and enhanced dissipation for the Kolmogorov
flow},  Adv. Math., 362 (2020): 106963. https://doi.org/10.1016/j.aim.2019.106963.

\bibitem{Yaglom}
A. M. Yaglom, {\it Hydrodynamic instability and transition to turbulence}, Fluid Mechanics and Its Applications, 100, Springer, Dordrecht, 2012.

\end{thebibliography}
\end{document}